\documentclass{amsproc}
\usepackage{graphicx}
\usepackage{amscd}
\usepackage{amsmath}
\usepackage{amsfonts}
\usepackage{amssymb}
\theoremstyle{plain}

\newtheorem{definition}{Definition}

\newtheorem{remark}{Remark}

\newtheorem{theorem}{Theorem}
\numberwithin{equation}{section}

\begin{document}
\title[Sobolev inequalities and interpolation spaces]{Sobolev inequalities, rearrangements, isoperimetry and interpolation spaces}
\author{Joaquim Mart\'{\i}n$^{\ast}$}
\address{Department of Mathematics\\
Universitat Aut\`{o}noma de Barcelona}
\email{jmartin@mat.uab.cat}
\author{Mario Milman}
\address{Department of Mathematics\\
Florida Atlantic University\\
Boca Raton, Fl. 33431}
\email{extrapol@bellsouth.net}
\urladdr{http://www.math.fau.edu/milman}
\thanks{$^{\ast}$Partially supported in part by Grants MTM2007-60500, MTM2008-05561-C02-02.}
\thanks{This paper is in final form and no version of it will be submitted for
publication elsewhere.}
\thanks{2000 Mathematics Subject Classification Primary: 46E30, 26D10.}
\dedicatory{Dedicated to our friends Bj\"{o}rn Jawerth and Evgeniy Pustylnik on the
ocassion of their 130th birthday (57th and 73th birthdays, respectively).}

\begin{abstract}
We characterize Poincar\'{e} inequalities in metric spaces using rearrangement inequalities.
\end{abstract}\maketitle

\section{Introduction}

Our starting point is the classical Gagliardo-Nirenberg inequality
which states that, for $n>1,$ $\frac{1}{n^{\prime}}=1-\frac{1}{n},$
\begin{equation}
\left\|  f\right\|  _{n^{\prime}}\leq\tau_{n}^{-1}\left\|  \left|
\nabla
f\right|  \right\|  _{L^{1}},\text{ }f\in Lip_{0}(\mathbb{R}^{n}), \label{a1}%
\end{equation}
where $Lip_{0}(\mathbb{R}^{n})$ denotes the set of Lipschitz
function on $\mathbb{R}^{n}$ with compact support,
$\tau_{n}=n\beta_{n}^{1/n}$ and $\beta_{n}=$ volume of the unit ball
in $\mathbb{R}^{n}$. It is well known (cf. \cite{maz} and
\cite{flem}), that (\ref{a1}) is equivalent to the isoperimetric
inequality\footnote{Here $m$ stands for Lebesgue measure and $m^{+}$
for Minkowski's content.}: for all Borel sets $A$ with
$m(A)<\infty,$
we have%
\begin{equation}
\tau_{n}\left(  m(A)\right)  ^{1/n^{\prime}}\leq m^{+}(A). \label{isop1}%
\end{equation}

We argue that it is worthwhile to consider a slightly more general
problem. Let $X=X(\mathbb{R}^{n})$ be a rearrangement invariant
space\footnote{i.e. such that if $f$ and $g$ have the same
distribution function then $\left\| f\right\|  _{X}=\left\|
g\right\|  _{X}$ (see Section \ref{sec24} below).}:
We ask for necessary and sufficient conditions such that%
\begin{equation}
\left\|  f\right\|  _{X}\leq c\left\|  \left|  \nabla f\right|
\right\|
_{L^{1}},\text{ }f\in Lip_{0}(\mathbb{R}^{n}), \label{a3}%
\end{equation}
holds. Maz'ya's classical method already shows that the problem has
a remarkably simple solution: (\ref{a3}) holds if and only if there
exists a
constant $c=c(n)>0$ such that for all Borel sets $A$ with $m(A)<\infty,$%
\begin{equation}
\phi_{X}(m(A))\leq cm^{+}(A), \label{a4}%
\end{equation}
where $\phi_{X}(t)$ is the fundamental function\footnote{It is well
known and easy to see that $\phi_{X}$ is continous, increasing and
equivalent to a
concave function.} of $X:$%
\[
\phi_{X}(t)=\left\|  \chi_{A}\right\|  _{X},\text{ with }m(A)=t.
\]

Formally (abusing the notation), the implication (\ref{a3})
$\Rightarrow $(\ref{a4}) follows inserting ``$f=\chi_{A}"$ in
(\ref{a3}) and then computing
$\left\|  \nabla f\right\|  _{L_{1}}=m^{+}(A),$ $\left\|  f\right\|  _{X}%
=\phi_{X}(m(A)).$

We now consider the converse statement. Here it will become clear
why we insist to work within the class of rearrangement invariant
spaces: Indeed, if we fix before hand a specific subclass of
rearrangement invariant spaces (e.g. Orlicz spaces) we would miss a
remarkable self-improving phenomenon.

Let $f\in Lip_{0}(\mathbb{R}^{n})$, and let $A_{t}=\{\left|
f\right|  >t\},$ $m(A_{t})=m_{f}(t)$ $(=$ the distribution function
of $f),$ then, from
(\ref{a4}), and the co-area formula, we find that%
\begin{equation*}\begin{split}
\int_{0}^{\infty}\phi_{X}(m_{f}(t))dt  &  \leq c\int_{0}^{\infty}m^{+}%
(A_{t})dt=c\int_{\mathbb{R}^{n}}\left|  \nabla\left|  f\right|
(x)\right|
dx\\
&  \leq c\int_{\mathbb{R}^{n}}\left|  \nabla f(x)\right|  dx.
\end{split}\end{equation*}
The integral on the left hand side is, by definition, the norm of
$f$ in the Lorentz space $\Lambda(X)$ associated with $X,$
\[
\left\|  f\right\|
_{\Lambda(X)}=\int_{0}^{\infty}\phi_{X}(m_{f}(t))dt.
\]
$\Lambda(X)$ is contained (and, in general, strictly contained) in
$X;$ in other words we have (cf. \cite{bs})
\begin{equation}
\left\|  f\right\|  _{X}\leq\left\|  f\right\|  _{\Lambda(X)}.
\label{necesitao}%
\end{equation}
Altogether, we have thus shown that
\[
\left\|  f\right\|  _{X}\leq\left\|  f\right\|  _{\Lambda(X)}\leq
c\left\| \nabla f\right\|  _{L^{1}}.
\]
Therefore, for $f\in Lip_{0}(\mathbb{R}^{n}),$ we have the
remarkable self
improvement%
\[
\left\|  f\right\|  _{X}\leq c\left\|  \left|  \nabla f\right|
\right\| _{L^{1}}\Leftrightarrow\left\|  f\right\|
_{\Lambda(X)}\leq c\left\|  \left| \nabla f\right|  \right\|
_{L^{1}}.
\]
But we are not quite done yet. We could have obtained the same
result starting from a much weaker inequality. Indeed, there is
another natural rearrangement invariant space (r.i. space)
associated to $X$: the somewhat larger Marcinkiewicz space $M(X)$
($=$Marcinkiewicz$=$weak type space) defined by the
quasi-norm%
\[
\left\|  f\right\|  _{M(X)}=\sup_{t>0}f^{\ast}(t)\phi_{X}(t)=\sup_{t>0}%
t\phi_{X}(m_{f}(t)),
\]
where $f^{\ast}$ is the non-increasing
rearrangement\footnote{$f^{\ast}$ is the generalized inverse of
$m_{f}.$} of $f.$ The fundamental functions of
these spaces satisfy%
\begin{equation}
\phi_{M(X)}(t)=\phi_{\Lambda(X)}(t)=\phi_{X}(t). \label{funda1}%
\end{equation}
It follows that for $f\in Lip_{0}(\mathbb{R}^{n}),$%
\[
\left\|  f\right\|  _{M(X)}\leq c\left\|  \left|  \nabla f\right|
\right\| _{L^{1}}\Leftrightarrow\left\|  f\right\|
_{\Lambda(X)}\leq c\left\|  \left| \nabla f\right|  \right\|
_{L^{1}}.
\]
These self-improving results are best possible since the spaces
$\Lambda(X),$ $M(X)$ are respectively the smallest and largest r.i.
spaces with fundamental functions equal to $\phi_{X}(t)$ (cf.
(\ref{funda1})), and such that (cf. \cite{bs})
\[
\Lambda(X)\subset X\subset M(X).
\]

A consequence of our discussion is that the optimal spaces $X$ for
the embedding (\ref{a3}) must be Lorentz spaces.

We now develop a quantitative connection with Euclidean
isoperimetry. For this purpose it is important to consider the
isoperimetric profile of $\mathbb{R}^{n}$
\[
I(t)=\inf_{m(A)=t}m^{+}(A).
\]
The isoperimetric inequality (\ref{isop1}) is the statement that for
$n>1,$
$I(t)$ is given by%

\[
I(t)=\tau_{n}t^{1/n^{\prime}},\text{ \ \ \ \ }n^{\prime}=n/(n-1).
\]
Note that $\Lambda(L^{n^{\prime}})=L(n^{\prime},1):$%
\begin{align*}
\left\|  f\right\|  _{\Lambda(L^{n^{\prime}})}  &
=\int_{0}^{\infty}\left( m_{f}(t)\right)
^{1/n^{\prime}}dt=\int_{0}^{\infty}t^{1/n^{\prime}}df^{\ast
}(t)\\
&
=\frac{1}{n^{\prime}}\int_{0}^{\infty}t^{1/n^{\prime}}f^{\ast}(t)\frac
{dt}{t}=\frac{1}{n^{\prime}}\left\|  f\right\|  _{L(n^{\prime},1)}.
\end{align*}
Therefore the previous analysis shows that the Gagliardo-Nirenberg
inequality
(\ref{a1}) self improves to its sharper form%
\begin{equation}
\left\|  f\right\|  _{L(n^{\prime},1)}\leq
n^{\prime}\tau_{n}^{-1}\left\|
\left|  \nabla f\right|  \right\|  _{L^{1}}. \label{gnlo}%
\end{equation}

The results that underlie the narrative above, including the sharp
Gagliardo-Nirenberg inequality (\ref{gnlo}), are, of course, well
known. But the added generality becomes more illuminating when we
move away from the classical Euclidean setting. Indeed, the argument
that gives the equivalence (\ref{a3}) $\Leftrightarrow$ (\ref{a4})
is very general and holds replacing $\mathbb{R}^{n}$ by fairly
general metric measure spaces as long as we have a suitable co-area
formula (cf. Bobkov-Houdr\'{e} \cite{bobk}, Coulhon \cite{coul} and
the references therein).

Consider a connected, metric, non-atomic measure space
$(\Omega,d,\mu).$ For a Lipschitz function $f$ on $\Omega$ we let
$|\nabla f(x)|=\lim\sup
_{d(x,y)\rightarrow0}\frac{|f(x)-f(y)|}{d(x,y)},$ and let
$Lip_{0}(\Omega)$ denote the Lipschitz functions with compact
support. Let us further assume that the equivalence between
\begin{equation}
\left\|  f\right\|  _{X}\leq c\left\|  |\nabla f|\right\|
_{L^{1}},\text{
\ }f\in Lip_{0}(\Omega) \label{a3'}%
\end{equation}
and
\begin{equation}
\phi_{X}(\mu(A))\leq c\mu^{+}(A), \label{condicion}%
\end{equation}
holds\footnote{here $\mu(A)<\infty,$ $\mu^{+}(A)$=perimeter of $A$
(see Section \ref{necesitao}\ below).}. We suppose, moreover, that
the associated
isoperimetric profile $I=I_{\Omega},$ defined by%
\[
I(t)=\inf_{\mu(A)=t}\mu^{+}(A)
\]
is continuous, increasing and concave. The same analysis then shows
that the best possible r.i. space such that (\ref{a3'}) holds is a
Lorentz space and its corresponding fundamental function $\phi,$
say, must be such that (\ref{condicion}) holds. The optimal space
corresponds to choosing the largest possible $\phi$ that satisfies
(\ref{condicion}), consequently the best choice
is $\phi=I=I_{\Omega}!$ Therefore we have%
\begin{equation}
\left\|  f\right\|  _{\Lambda(I)}\leq\left\|  |\nabla f|\right\|
_{L^{1}},
\label{logsob1}%
\end{equation}
where $\Lambda(I)$ is ``the isoperimetric Lorentz space'' defined by
\begin{equation}
\left\|  f\right\|  _{\Lambda(I)}=\int_{0}^{\infty}I(\mu_{f}(t))dt.
\label{definide}%
\end{equation}

General Sobolev inequalities, including Logarithmic Sobolev
inequalities, fit into this picture very naturally. Indeed, in this
fashion we have a natural method to construct best possible Sobolev
inequalities if we understand the isoperimetry associated with a
given geometry.

It is worthwhile to discuss in some detail how this point of view
applies to Gaussian measure (cf. \cite{mmlog}). In the Gaussian
world the isoperimetric function $I$ has the following properties:
$I$ is defined on $[0,1],$ it is increasing on $[0,1/2],$ symmetric
about $1/2,$ and $I$ is concave. Since we are dealing with a
probability space, from the point of view of describing the
underlying function spaces it is only important to know the behavior
of $I$ near the origin. We actually have\footnote{Here the symbol
$f\simeq g$ indicates the existence of a universal constant $c>0$
(independent of all parameters involved) such that $(1/c)f\leq g\leq
c\,f$. Likewise the symbol $f\preceq g$ will mean that there exists
a universal constant $c>0$
(independent of all parameters involved) such that $f\leq c\,g$.}%
\[
I(t)\simeq t\left(  \log\frac{1}{t}\right)  ^{1/2},\text{ for }t\in
\lbrack0,1/2].
\]
In this case (\ref{definide}) is not a norm but nevertheless the set
of all $f$ with $\left\|  f\right\|  _{\Lambda(I)}<\infty$ is
equivalent to the Lorentz space $L(LogL)^{1/2}:$ In other words, as
sets,
\[
\Lambda(I)=L(LogL)^{1/2}.
\]
In this setting the inequality (\ref{logsob1}), which is due to
Ledoux \cite{le}, can be seen as part of the usual family of Log
Sobolev inequalities. Thus, in the Gaussian world, Ledoux's
inequality plays the role of the classical (Euclidean) sharp
Gagliardo-Nirenberg inequality.

More generally, the ``isoperimetric Lorentz spaces'' can be used to
construct the corresponding Gagliardo-Nirenberg inequalities in
other geometries.

Let us mention two obvious drawbacks of the previous discussion: (a)
we only considered Sobolev spaces where the gradient is in $L^{1}$,
(b) the analysis is *space dependent*. On the other hand, already in
the Euclidean case, Maz'ya showed that ``all $L^{p}$ Sobolev''
inequalities can be obtained from the isoperimetric inequality or,
equivalently, from (\ref{a1}). In our recent work we have considered
the extension of Maz'ya's ideas to rearrangement invariant spaces.

Maz'ya's smooth truncation method has been extensively studied in
the literature (cf. \cite{bakr}, \cite{haj}, and the references
therein) but in our development we required an extension that leads
to pointwise rearrangement\footnote{also called ``symmetrization''
inequalities since they are often expressed in terms of
``symmetric'' rearrangements.} inequalities that depend on the
isoperimetric profile. For example, we showed in a very general
setting (cf. \cite{mmp}, \cite{mmlog}, \cite{mmadv}) inequalities of
the form
\begin{equation}
f_{\mu}^{\ast\ast}(t)-f_{\mu}^{\ast}(t)\leq\frac{t}{I(t)}\left|
\nabla f\right|  _{\mu}^{\ast\ast}(t),\text{ }f\in Lip(\Omega)\cap
L^{1}\left(
\Omega\right)  , \label{pro1}%
\end{equation}
where
$f_{\mu}^{\ast\ast}(t)=\frac{1}{t}\int_{0}^{t}f_{\mu}^{\ast}(s)ds,$
and $f_{\mu}^{\ast}$ is the non increasing rearrangement of $f$ with
respect to the measure $\mu$ on ${\Omega}$ (see Section \ref{sec24}
below). Let us now show in some detail that (\ref{pro1}) implies the
isoperimetric inequality
(cf. \cite{mmadv}). Following \cite{bobk} we select a sequence $\{f_{n}%
\}_{n\in N}$ in $Lip(\Omega)\cap L^{1}\left(  \Omega\right)  $, such
that
$f_{n}\underset{L^{1}}{\rightarrow}\chi_{A}$, and%
\begin{equation}
\mu^{+}(A)\geq\lim\sup_{n\rightarrow\infty}\left\|  \left|  \nabla
f_{n}\right|  \right\|  _{L^{1}}. \label{notada}%
\end{equation}
Let $t>\mu(A)$ and apply (\ref{pro1}) to this sequence. We have
\[
\left(  f_{n}\right)  _{\mu}^{\ast\ast}(t)-\left(  f_{n}\right)
_{\mu}^{\ast }(t)\leq\frac{t}{I(t)}\left|  \nabla f_{n}\right|
_{\mu}^{\ast\ast}(t),\text{ }n\in N.
\]
By definition%
\begin{equation*}\begin{split}
t\left|  \nabla f_{n}\right|  _{\mu}^{\ast\ast}(t)  &
=\int_{0}^{t}\left|
\nabla f_{n}\right|  _{\mu}^{\ast}(s)ds\\
&  \leq\left\|  \left|  \nabla f_{n}\right|  \right\|  _{L^{1}}.
\end{split}\end{equation*}
Therefore,
\[
\lim\sup_{n}t\left|  \nabla f_{n}\right|
_{\mu}^{\ast\ast}(t)\leq\lim \sup_{n\rightarrow\infty}\left\|
\left|  \nabla f_{n}\right|  \right\| _{L^{1}}\leq\mu^{+}(A).
\]
On the other hand by \cite{gar} we have
\[
I(t)\left(  \left(  f_{n}\right)  _{\mu}^{\ast\ast}(t)-\left(
f_{n}\right)
_{\mu}^{\ast}(t)\right)  \rightarrow I(t)\left(  \chi_{A}^{\ast\ast}%
(t)-\chi_{A}^{\ast}(t)\right)  .
\]
Combining our findings we have%
\begin{equation}
I(t)\left(  \chi_{A}^{\ast\ast}(t)-\chi_{A}^{\ast}(t)\right)
\leq\mu
^{+}(A),\text{ for all }t>\mu(A). \label{insertao}%
\end{equation}
Now, since $\chi_{A}^{\ast}=\chi_{(0,\mu(A))},$ we have that for
$t>\mu(A),$
\[
\chi_{A}^{\ast}(t)=\chi_{(0,\mu(A))}(t)=0,\chi_{A}^{\ast\ast}(t)=\frac{1}%
{t}\int_{0}^{t}\chi_{(0,\mu(A))}(s)ds=\frac{\mu(A)}{t}.
\]
Inserting this information in (\ref{insertao}) we get%
\[
I(t)\frac{\mu(A)}{t}\leq\mu^{+}(A).
\]
Finally we let $t\rightarrow\mu(A);$ then, by the continuity of $I,$
we obtain
the isoperimetric inequality%
\[
I(\mu(A))\leq\mu^{+}(A).
\]

We now discuss the corresponding Sobolev inequalities with $L^{q},$
$q>1,$ replacing the $L^{1}$ norm on the right hand side of
(\ref{a3'}). Again we shall work on suitable metric probability
spaces $(\Omega,d,\mu)$\footnote{for a list of the assumptions and
further background information see Section 2.},
and we consider Poincar\'{e} inequalities of the form%
\begin{equation}
\left\|  f-m(f)\right\|  _{X}\leq c\left\|  |\nabla f|\right\|  _{L^{q}%
},\text{ }f\in Lip(\Omega),\text{ }q>1, \label{cap1}%
\end{equation}
where $X$ is a r.i. space and $m(f)$ is a median\footnote{a real
number $m(f)$ such that $\mu\left\{  f\geq m(f)\right\}
\geq1/2\text{ \ and }\mu\left\{ f\leq m(f)\right\}  \geq1/2.$} of
$f.$ As is well known, inequalities of this type can be
characterized using Maz'ya's theory of capacities (cf. \cite{ma}).
The weak type version of (\ref{cap1}) reads:
\begin{equation}
\left\|  f-m(f)\right\|  _{M(X)}\preceq\left\|  |\nabla f|\right\|  _{L^{q}%
},\text{ }f\in Lip(\Omega),\text{ }q>1. \label{ema1}%
\end{equation}
In this context a result of E. Milman \cite[Proposition
3.8]{emanuel} can be
rewritten in our notation as saying that (\ref{ema1}) is equivalent to%
\begin{equation}
\left(  \phi_{X}(t)\right)  ^{q}\preceq cap_{q}(t,1/2),\text{
}0<t<1/2,
\label{ema2}%
\end{equation}
where (using temporarily\footnote{See Definition \ref{definida}
below.} the definition of \cite{bobz} rather than the one in
\cite{emanuel})
\[
cap_{q}(t,1/2)=\inf\{\left\|  |\nabla\Phi|\right\|  _{L^{q}}^{q}:\mu
\{\Phi=1\}\geq t,\mu\{\Phi=0\}\geq1/2\},
\]
and the infimum is taken over all
$\Phi:\Omega\rightarrow\lbrack0,1]$ that are Lipschitz on balls.

To relate (\ref{ema2}) to $X$ norm inequalities we use the
$\Lambda_{q}(X)$
Lorentz spaces defined by%
\[
\left\|  f\right\|  _{\Lambda_{q}(X)}=\left(  \int_{0}^{\infty}\phi_{X}%
(\mu_{f}(t))dt^{q}\right)  ^{1/q}.
\]
We say that $X$ is $q-$\textbf{concave} if the space $X_{(q)}$
defined by:
\[
X_{(q)}=\{f:\left|  f\right|  ^{1/q}\in X\},\text{ \ \ }\left\|
f\right\| _{X_{(q)}}=\left\|  \left|  f\right|  ^{1/q}\right\|
_{X}^{q},
\]
is a r.i. space (see Section \ref{sec24} below). Moreover, it
follows from the
definitions that%
\begin{equation}
\left\|  f\right\|  _{X}^{q}=\left\|  \left|  f\right|  ^{q}\right\|
_{X_{(q)}}.\label{angelico}%
\end{equation}
We develop this theme for the $\Lambda_{q}(X)$ scale in detail. From
$\mu_{\left|  f\right|  ^{q}}(t)=\mu_{\left|  f\right|  }(t^{1/q}),$
we see
that\footnote{in other words $\Lambda_{q}(X)_{(q)}=\Lambda(X).$}%
\[
\left\|  \left|  f\right|  ^{q}\right\|  _{\Lambda(X)}=\left\|
f\right\|
_{\Lambda_{q}(X)}^{q}%
\]
in particular for $X_{(q)}$ we have%
\begin{equation}
\left\|  \left|  f\right|  ^{q}\right\|  _{\Lambda(X_{(q)})}=\left\|
f\right\|  _{\Lambda_{q}(X_{(q)})}^{q}.\label{angelica}%
\end{equation}
Note that for any measurable set with $\mu(A)=t,$ we have%
\[
\phi_{X_{(q)}}(t)=\left\|  \chi_{A}\right\|  _{X_{(q)}}=\left\| \chi
_{A}\right\|  _{X}^{q}=\left(  \phi_{X}(t)\right)  ^{q},
\]
in particular if $X$ is $q-$concave the function
$\phi_{X_{(q)}}(t)=\left(
\phi_{X}(t)\right)  ^{q}$ is concave, and (\ref{ema2}) now reads%
\begin{equation}
\phi_{X_{(q)}}(t)\preceq cap_{q}(t,1/2),\text{ }0<t<1/2.\label{ema2'}%
\end{equation}

Thus, using the characterization of Sobolev norms in terms of
capacities, due to Maz'ya (in the form given by Bobkov and
Zegarlinski for metric paces
\cite[Lemma 5.6]{bobz}), we now show that (\ref{ema1}) self improves to%
\begin{equation}
\left\|  f-m(f)\right\|  _{\Lambda_{q}(X_{(q)})}\preceq\left\|
|\nabla
f|\right\|  _{L^{q}}.\label{ema3}%
\end{equation}
To see this we use (\ref{ema2'}) as follows. First we observe that
it is enough to prove (\ref{ema3}) for positive functions that are
Lipschitz on balls, such that $\left\|  f\right\|  _{\infty}\leq1,$
and, moreover, such that $m(f)=0$ (see details of the argument that
proves this assertion in \cite[page 331]{emanuel}). Let $f$ be a
function satisfying all these
conditions, then, by (\ref{ema2'}), we have%
\[
\phi_{X_{(q)}}(\mu_{f}(t))\preceq cap_{q}(\mu_{f}(t),1/2).
\]
Therefore%
\begin{equation*}\begin{split}
\left\|  f-0\right\|  _{\Lambda_{q}(X_{(q)})}^{q} &  =\int_{0}^{\infty}%
\phi_{X_{(q)}}(\mu_{f}(t))dt^{q}\preceq\int_{0}^{\infty}cap_{q}(\mu
_{f}(t),1/2)dt^{q}\\
&  \preceq\left\|  \left|  \nabla f\right|  \right\|  _{L^{q}}^{q},
\end{split}\end{equation*}
where the last inequality follows from Bobkov and Zegarlinski
\cite[Lemma 5.6]{bobz} changing 2 for $q$ in the argument given
there.

Finally, combining with (\ref{necesitao}), (\ref{angelico}) and
(\ref{angelica}), we obtain%
\begin{equation*}\begin{split}
\left\|  f\right\|  _{X}^{q}  & =\left\|  \left|  f\right|
^{q}\right\|
_{X_{(q)}}\\
& \leq\left\|  \left|  f\right|  ^{q}\right\|  _{\Lambda(X_{(q)})}\\
& =\left\|  f\right\|  _{\Lambda_{q}(X_{(q)})}^{q}\\
& \preceq\left\|  \left|  \nabla f\right|  \right\|  _{L^{q}}^{q}.
\end{split}\end{equation*}
Thus, we see that the Sobolev self improvement that we obtained in
the case $q=1$ extends to the case $q>1$, but now it is expressed in
terms of the $\Lambda_{q}(X_{(q)})$ spaces. More precisely, for $q-$
concave spaces we have the following equivalences on $Lip$ functions
\begin{equation*}\begin{split}
\left\|  f-m(f)\right\|  _{M(X)} &  \preceq\left\|  \left|  \nabla
f\right| \right\|  _{L^{q}}\Leftrightarrow\left\|  f-m(f)\right\|
_{\Lambda
_{q}(X_{(q)})}\preceq\left\|  \left|  \nabla f\right|  \right\|  _{L^{q}}\\
&  \Leftrightarrow\left\|  f-m(f)\right\|  _{X}\preceq\left\| \left|
\nabla f\right|  \right\|  _{L^{q}}.
\end{split}\end{equation*}

After this lengthy introduction we now describe the purpose of this
note. We shall consider the analogues of the rearrangement
inequalities (\ref{pro1}) that correspond to consider homogenous
Sobolev norms with $q>1$ on the right hand side$.$ The inequalities
we shall obtain will be naturally formulated in terms of the
$q-$convexification $X^{(q)}$ of $X$ (see (\ref{scalaq}) in Section
\ref{sec24} below).  We also pay close attention to the basic
assumptions that one needs to place on the isoperimetric profile,
and the probability measure spaces, in order to develop a meaningful
theory with mild assumptions. In particular, we are able to extend
some results of \cite{mmadv} under weaker assumptions.

Finally in Section 3 we shall briefly discuss a connection with
interpolation theory, that was recently developed in \cite{cjm},
that shows a larger context for the Sobolev oscillation inequalities
and connects some aspects of our work with the theory of
extrapolation of martingale inequalities.

\section{Capacitary Inequalities}

\subsection{Background\label{necesitao}}

From now on ``a metric probability space $\left(
\Omega,d,\mu\right)  $'' will be a connected separable metric space
$\left(  \Omega,d,\mu\right)  $ equipped with a non-atomic Borel
probability measure $\mu$. For measurable functions
$u:\Omega\rightarrow\mathbb{R},$ the distribution function of $u$ is
given by
\[
\mu_{u}(t)=\mu\{x\in{\Omega}:\left|  u(x)\right|  >t\}\text{ \ \ \ \
}(t>0).
\]
The \textbf{decreasing rearrangement} $u_{\mu}^{\ast}$ of $u$ is the
right-continuous non-increasing function from $(0,1)$ to
$[0,\infty]$ which is equimeasurable with $u$. Namely,
\[
u_{\mu}^{\ast}(s)=\inf\{t\geq0:\mu_{u}(t)\leq s\}.
\]
We have (cf. \cite{bs}),%
\begin{equation}
\sup_{\mu(E)\leq t}\int_{E}\left|  u(x)\right|
d\mu(x)=\int_{0}^{\mu
(E)}u_{\mu}^{\ast}(s)ds. \label{hp}%
\end{equation}
Since $u_{\mu}^{\ast}$ is decreasing, the function
$u_{\mu}^{\ast\ast},$ defined for integrable functions by
\[
u_{\mu}^{\ast\ast}(t)=\frac{1}{t}\int_{0}^{t}u_{\mu}^{\ast}(s)ds,
\]
is also decreasing and, moreover,
\[
u_{\mu}^{\ast}\leq u_{\mu}^{\ast\ast}.
\]

As customary, if $A\subset\Omega$ is a Borel set$,$ the
\textbf{perimeter} or \textbf{Minkowski content} of $A$ is defined
by
\[
\mu^{+}(A)=\lim\inf_{h\rightarrow0}\frac{\mu\left(  A_{h}\right)
-\mu\left( A\right)  }{h},
\]
where $A_{h}=\left\{  x\in\Omega:d(x,A)<h\right\}  .$

The \textbf{isoperimetric profile} $I_{(\Omega,d,\mu)}$ is defined
as the pointwise maximal function
$I_{(\Omega,d,\mu)}:[0,1]\rightarrow\left[
0,\infty\right)  $ such that%
\[
\mu^{+}(A)\geq I_{(\Omega,d,\mu)}(\mu(A)),
\]
holds for all Borel sets $A$.

For a Lipschitz function $f$ on $\Omega$ (briefly $f\in
Lip(\Omega))$ we define, as usual, the \textbf{modulus of the
gradient} by
\[
|\nabla
f(x)|=\lim\sup_{d(x,y)\rightarrow0}\frac{|f(x)-f(y)|}{d(x,y)}.
\]

One of the themes of our recent paper \cite{mmadv} was to
characterize generalized Gagliardo-Nirenberg inequalities and
Poincar\'{e} inequalities using rearrangement inequalities. The
setting of \cite{mmadv} were metric probability spaces $\left(
\Omega,d,\mu\right)  $ that satisfy the following conditions:

\textbf{Condition 1: }The isoperimetric profile $I_{(\Omega,d,\mu)}$
is a concave continuous function, increasing on $(0,1/2),$ symmetric
about the point $1/2$ such that, moreover, vanishes at zero.

\begin{remark}
Condition 1 played an important role in the formulation of the
inequalities obtained in \cite{mmadv}. In this note we shall show
that, suitably reformulated (cf. \ref{reod00} below), our
inequalities remain true under the weaker Condition 1' below.
\end{remark}

\textbf{Condition 2: }For every $f\in Lip(\Omega)$ , and every $c\in
R$, we have that $|\nabla f(x)|=0$, $\mu-$a.e. on the set
$\{x:f(x)=c\}$.

\begin{remark}
Condition 2 is used to compare the gradients of Lip functions that
coincide on a given set, which is particularly useful to deal with
truncations. Moreover, it implies that $\int_{\{f=t\}}\left|  \nabla
f\right|  d\mu=0,$ even on sets where may have $\mu\{f=t\}>0.$ Using
an approximation argument of E. Milman \cite[Remark 3.3]{Emil01} we
will show how to dispense with this condition as well (cf. Theorem
\ref{capa03} below).
\end{remark}

In this paper, we consider Sobolev inequalities for $q\geq1,$
moreover, following a suggestion of Michel Ledoux, we shall impose
weaker restrictions on the metric spaces. More specifically, we will
eliminate \textbf{Condition 2} and replace Condition 1 with the
following much weaker assumption

\textbf{Condition 1': }The isoperimetric profile
$I_{(\Omega,d,\mu)}$ is a positive continuous function that vanishes
at zero.

\begin{remark}
Notice that the continuity assumption, and (\ref{mazya2}),
(\ref{sime}) below, imply that $I$ is symmetric about the point
$1/2$ (see \cite[Corollary 6.5]{MiE}). Moreover, we see that for
$q>1$ the function $\frac{1}{\left( \inf_{t\leq z\leq1/2}I(z)\right)
^{\frac{q}{q-1}}}$ is locally integrable on $(0,1).$
\end{remark}

The notion of capacity plays a fundamental role in the theory
developed by V. G. Maz'ya and his school to study functional
inequalities and embedding theorems (see \cite{ma}). For the study
of capacities in metric spaces we also refer to see \cite{barth},
\cite{bobz}, \cite{emanuel}, and the references therein). Capacities
will also play a decisive role in our development in this note.

\begin{definition}
\label{definida}Let $\left(  \Omega,d,\mu\right)  $ be a metric
probability space, and let $1\leq q<\infty.$ Given two Borel sets
$A\subset$ $B\subset\Omega$, the $q-$capacity of $A$ relative to $B$
is defined by
\[
Cap_{q}(A,B)=\inf\left\{  \left\|  |\nabla\Phi|\right\|
_{L^{q}}:\Phi_{\mid A}=1,\text{ }\Phi_{\mid\Omega\setminus
B}=0\right\}  ,
\]
where the infimum is over all $\Phi:\rightarrow\lbrack0,1]$ which
are Lipschitz-on-balls.

Let $0<a\leq b<1,$ the $q-$capacity profile is defined by
\begin{equation*}\begin{split}
cap_{q}(a,b)  &  =\inf\left\{  Cap_{q}(A,B):A\subset B,\text{ \
}\mu\left\{
A\right\}  \geq a,\text{ }\mu\left\{  B\right\}  \leq b\right\} \\
&  =\inf\left\{  \left\|  | \nabla\Phi|\right\|  _{L^{q}}:\mu\left\{
\Phi=1\right\}  \geq a,\text{ }\mu\left\{  \Phi=0\right\}
\geq1-b\right\}  ,
\end{split}\end{equation*}
where the latter infimum is taken over all
$\Phi:\rightarrow\lbrack0,1]$ which are Lipschitz-on-balls.
\end{definition}

Let us also recall some properties concerning capacities that will
be useful in what follows:

\begin{enumerate}
\item  It is plain from the definition that
\begin{equation}
cap_{q}(a,b)=cap_{q}(1-b,1-a),\ \ (0<a\leq b<1). \label{sime}%
\end{equation}
Moreover, the functional $a\rightarrow cap_{q}(\cdot,b)$ is
increasing; and $b\rightarrow cap_{q}(a,\cdot)$ is decreasing.

\item (See \cite[p. 105]{ma} and \cite{emanuel}) Let $1<q<\infty,$ then%
\begin{equation}
\frac{1}{cap_{q}(a,b)}\leq\left(
\int_{a}^{b}\frac{ds}{cap_{1}(s,b)^{\frac {q}{q-1}}}\right)
^{\frac{q-1}{q}},\text{ \ \ (}0<a\leq b<1\text{).}
\label{mazya1}%
\end{equation}

\item (See \cite{Maz01}, \cite{flem}, \cite{bobk}, and the references therein)
The connection between the $1-$capacity and the isoperimetric
profile is given
by:%
\[
\inf_{a\leq t\leq b}I(t)\leq cap_{1}(a,b)\leq\inf_{a\leq
t<b}I(t);\text{ \ \ }\left(  0<a<b<1\right)  .
\]
Therefore, since we assume the continuity of the isoperimetric
profile $I,$ we have
\begin{equation}
\inf_{a\leq t\leq b}I(t)=cap_{1}(a,b);\text{ \ \ }\left(
0<a<b<1\right)  .
\label{mazya2}%
\end{equation}

\item  Combining (\ref{mazya1}) and (\ref{mazya2}) we get
\begin{equation}
\frac{1}{cap_{q}(a,b)}\leq\left(  \int_{a}^{b}\frac{ds}{\left(
\inf_{s\leq t\leq b}I(t)\right)  ^{\frac{q}{q-1}}}\right)
^{\frac{q-1}{q}}.
\label{mazya3}%
\end{equation}
\end{enumerate}

Our main result will be formulated using following functions:

\begin{definition}
Let $I=I_{\left(  \Omega,d,\mu\right)  }$ be the isoperimetric
profile of
$\left(  \Omega,d,\mu\right)  ,$ and let $1\leq q<\infty.$ We let%
\[
w_{q}(t)=\left\{
\begin{array}
[c]{cc}%
\left(  \frac{1}{t}\int_{0}^{t}\left(  \frac{s}{I(s)}\right)  ^{\frac{q}{q-1}%
}ds\right)  ^{\frac{1-q}{q}} & \text{if }q>1\\
\inf_{0<s<t}\frac{I(s)}{s} & \text{if }q=1.
\end{array}
\right.
\]
\end{definition}

\begin{remark}
Notice that%
\begin{equation}
w_{1}(t)\leq w_{q_{1}}(t)\leq w_{q_{2}}(t)\text{ \ \ \ \ (}q_{1}\leq
q_{2}).
\label{orden}%
\end{equation}
Moreover, if $I(t)/t$ is decreasing, then
\begin{equation}
\frac{I(t)}{t}=w_{1}(t). \label{orden1}%
\end{equation}
\end{remark}

\subsection{Symmetrization inequalities under weak assumptions on the
isoperimetric profiles}

\begin{theorem}
\label{capa01}Let $\left(  \Omega,d,\mu\right)  $ be a metric
probability space that satisfies Conditions 1' and 2, and let $1\leq
q<\infty.$ Then for $f\in Lip(\Omega)\cap L^{1}\left(  \Omega\right)
,$ and for all $t\in(0,1),$ we have

\begin{enumerate}
\item
\begin{equation}
\int_{0}^{t}\left[  \left(  \left(  -f_{\mu}^{\ast}\right)  ^{\prime}%
(\cdot)I(\cdot)\right)  ^{\ast}(s)\right]
^{q}ds\leq\int_{0}^{t}\left(
\left|  \nabla f\right|  _{\mu}^{\ast}\right)  ^{q}(s)ds. \label{aa}%
\end{equation}

\item
\begin{equation}
(f_{\mu}^{\ast\ast}(t)-f_{\mu}^{\ast}(t))w_{q}(t)\leq\left(  \frac{1}{t}%
\int_{0}^{t}\left(  \left|  \nabla f\right|  _{\mu}^{\ast}\right)
^{q}(s)ds\right)  ^{1/q}. \label{reod00}%
\end{equation}
\end{enumerate}
\end{theorem}

\begin{remark}
Since $w_{1}(t)=\inf_{0<s<t}\frac{I(s)}{s}\leq I(t),$ it follows
readily that, for $q=1,$ the inequality (\ref{reod00}) is weaker
than (\ref{pro1}). On the other hand, (\ref{pro1}) was proved in
\cite{mmadv} under the stronger assumption that $I(t)$ is concave.
Now, if $I(t)$ is concave then $\frac {I(t)}{t}$ is decreasing;
therefore we have that $w_{1}(t)=I(t)$ (cf. (\ref{orden1})) and
consequently (\ref{reod00}) coincides with (\ref{pro1}).
\end{remark}

\begin{remark}
We do not consider here the corresponding problem of characterizing
(\ref{reod00}) (resp. (\ref{aa})) for $q>1$ in terms of
isocapacitary inequalities.
\end{remark}

\begin{proof}
Since $f\in Lip(\Omega)$ implies that $\left|  f\right|  \in
Lip(\Omega),$ and, moreover,
\[
\left|  \nabla f(x)\right|  \geq\left|  \nabla\left|  f\right|
(x)\right|  ,
\]
we can assume without loss of generality that $f\geq0.$

Let us start by proving that $f_{\mu}^{\ast}$ locally absolutely
continuous. The proof here follows very closely the one given in
\cite{mmadv} under the assumption that Condition 1 above holds.
Therefore, we will only indicate in
detail the changes that are required. Let $0<t_{1}<t_{2}<\infty$, and define%
\[
f_{t_{1}}^{t_{2}}(x)=\left\{
\begin{array}
[c]{ll}%
t_{2}-t_{1} & \text{if }f(x)\geq t_{2},\\
f(x)-t_{1} & \text{if }t_{1}<f(x)<t_{2},\\
0 & \text{if }f(x)\leq t_{1}.
\end{array}
\right.
\]
and let%

\[
N[f_{t_{1}}^{t_{2}}(x)]=\frac{f_{t_{1}}^{t_{2}}(x)}{t_{2}-t_{1}}.
\]
It follows that
\begin{equation}
N[f_{t_{1}}^{t_{2}}(x)]\text{ is }\left\{
\begin{array}
[c]{ll}%
=1 & \text{if }f(x)\geq t_{2},\\
<1 & \text{if }t_{1}<f(x)<t_{2},\\
=0 & \text{if }f(x)\leq t_{1}.
\end{array}
\right.  \label{definidanueva}%
\end{equation}
By Condition 2,
\[
\left|  \nabla N[f_{t_{1}}^{t_{2}}(x)]\right|
=\frac{1}{t_{2}-t_{1}}\left| \nabla f\right|  \chi_{\left\{
t_{1}<\left|  f\right|  <t_{2}\right\}
}\text{ }\mu-\text{a.e,}%
\]
and we have
\begin{align}
cap_{1}\left(  \mu\{\left|  f(x)\right|  \geq t_{2}\},\text{
}\mu\{\left| f(x)\right|  >t_{1}\}\right)   &  \leq Cap_{1}\left(
\{\left|  f(x)\right|
\geq t_{2}\},\text{ }\{\left|  f(x)\right|  >t_{1}\}\right) \nonumber\\
&  \leq\int_{\Omega}\left|  \nabla N[f_{t_{1}}^{t_{2}}(x)]\right|
d\mu(x)\nonumber\\
&  =\frac{1}{(t_{2}-t_{1})}\int_{\left\{  t_{1}<f<t_{2}\right\}
}\left|
\nabla f(x)\right|  d\mu(x). \label{aqui}%
\end{align}
Let $0<a<b,$ $t_{1}=f_{\mu}^{\ast}(b),$ $t_{2}=f_{\mu}^{\ast}(a),$
then (\ref{aqui}) yields
\begin{equation*}\begin{split}
&  cap_{1}\left(  \mu\left\{  \left|  f(x)\right|  \geq
f_{\mu}^{\ast }(a)\right\}  ,\text{ }\mu\left\{  \left|  f(x)\right|
>f_{\mu}^{\ast
}(b)\right\}  \right)  [f_{\mu}^{\ast}(a)-f_{\mu}^{\ast}(b)]\\
&  \leq\int_{\{f_{\mu}^{\ast}(b)<\left|  f\right|  <f_{\mu}^{\ast}%
(a)\}}\left|  \nabla f(x)\right|  d\mu(x).
\end{split}\end{equation*}
Since%
\[
a\leq\mu\left\{  \left|  f(x)\right|  \geq f_{\mu}^{\ast}(a)\right\}
\text{ and }\mu\left\{  \left|  f(x)\right|
>f_{\mu}^{\ast}(b)\right\}  <b,
\]
and $cap_{q}(.,.)$ is increasing in the first variable and
decreasing in the second, we see that
\begin{equation}
cap_{1}\left(  a,b\right)  [f_{\mu}^{\ast}(a)-f_{\mu}^{\ast}(b)]\leq
\int_{\{f_{\mu}^{\ast}(b)<\left|  f\right|
<f_{\mu}^{\ast}(a)\}}\left|
\nabla f(x)\right|  d\mu(x). \label{capa}%
\end{equation}
Let us see that $f_{\mu}^{\ast}$ is locally absolutely continuous.
Let us consider an interval $[a,b],0<a<b<1.$ Let $\{\left(
a_{k},b_{k}\right) \}_{k=1}^{r}$ be any finite family of
non-overlapping sub-intervals of $[a,b]$
such that $\sum_{k=1}^{r}\left(  b_{k}-a_{k}\right)  \leq\delta.$ We have%
\begin{equation*}\begin{split}
\mu\left\{  \cup_{k=1}^{r}\left\{  f_{\mu}^{\ast}(b_{k})<\left|
f\right| <f_{\mu}^{\ast}(a_{k})\right\}  \right\}   &
=\sum_{k=1}^{r}\mu\left\{
f_{\mu}^{\ast}(b_{k})<\left|  f\right|  <f_{\mu}^{\ast}(a_{k})\right\} \\
&  \leq\sum_{k=1}^{r}\left(  b_{k}-a_{k}\right)  \leq\delta.
\end{split}\end{equation*}
On the other hand, by (\ref{capa}), we have%
\begin{equation*}\begin{split}
\sum_{k=1}^{r}\left(
f_{\mu}^{\ast}(a_{k})-f_{\mu}^{\ast}(b_{k})\right)
cap_{1}(a_{k},b_{k})  &  \leq\sum_{k=1}^{r}\int_{\left\{  f_{\mu}^{\ast}%
(b_{k})<\left|  f\right|  <f_{\mu}^{\ast}(a_{k})\right\}  }\left|
\nabla\left|  f\right|  (x)\right|  d\mu(x)\\
&  =\int_{\cup_{k=1}^{r}\left\{  f_{\mu}^{\ast}(b_{k})<\left|
f\right| <f_{\mu}^{\ast}(a_{k})\right\}  }\left|  \nabla\left|
f\right|  (x)\right|
d\mu(x)\\
&  \leq\int_{0}^{\delta}\left|  \nabla\left|  f\right|  \right|
_{\mu}^{\ast
}(t)dt\\
&  \leq\int_{0}^{\delta}\left|  \nabla f\right|  _{\mu}^{\ast}(t)dt.
\end{split}\end{equation*}

We also observe that%
\[
cap_{1}(a_{k},b_{k})\geq cap_{1}(a,b_{k})=cap_{1}(1-b_{k},1-a)\geq
cap_{1}(1-b,1-a).
\]
Thus, combining our estimates we see that
\[
cap_{1}(1-b,1-a)\sum_{k=1}^{r}\left(
f_{\mu}^{\ast}(a_{k})-f_{\mu}^{\ast }(b_{k})\right)
\leq\int_{0}^{\delta}\left|  \nabla f\right|  _{\mu}^{\ast }(t)dt.
\]
The local absolute continuity of $f_{\mu}^{\ast}$ follows.

In the course of the proof of this theorem we shall also need to
know the local absolute continuity of the function
$\Psi(t)=\int_{\{\left|  f\right|
>f_{\mu}^{\ast}(t)\}}\left|  \nabla f(x)\right|  ^{q}d\mu(x),$ under the
assumption that $\left|  \nabla f(x)\right|  ^{q}\in L^{1}(\Omega).$
This fact can be easily seen with essentially with the same argument
we have just provided. Indeed, fix once again an interval
$[a,b]\subset(0,1),$ and consider
any finite family of non-overlapping sub-intervals $\{\left(  a_{k}%
,b_{k}\right)  \}_{k=1}^{r}$ of $[a,b]$ such that
$\sum_{k=1}^{r}\left(
b_{k}-a_{k}\right)  \leq\delta.$ We can then estimate as before%
\begin{equation*}\begin{split}
\sum_{k=1}^{r}\left|  \Psi(b_{k})-\Psi(a_{k})\right|   &  \leq\sum_{k=1}%
^{r}\int_{\left\{  f_{\mu}^{\ast}(b_{k})<\left|  f\right|
<f_{\mu}^{\ast
}(a_{k})\right\}  }\left|  \nabla f(x)\right|  ^{q}d\mu(x)\\
&  \leq\int_{0}^{\delta}\left|  \nabla f\right|  _{\mu}^{\ast
q}(t)dt,
\end{split}\end{equation*}
and the local absolute continuity of $\Psi$ follows.

We now prove (\ref{aa}).

Case $q>1.$ Let $0<h<t<1.$ The same argument that shows (\ref{capa})
yields with $a=t-h,b=t.$
\[
cap_{q}\left(  t-h,\text{ }t\right)
[f_{\mu}^{\ast}(t-h)-f_{\mu}^{\ast }(t)]\leq\left(
\int_{\{f_{\mu}^{\ast}(t)<\left|  f\right|  <f_{\mu}^{\ast
}(t-h)\}}\left|  \nabla f(x)\right|  ^{q}d\mu(x)\right)  ^{1/q}.
\]
Combining with (\ref{mazya3}) we obtain,
\[
\left[  \frac{f_{\mu}^{\ast}(t-h)-f_{\mu}^{\ast}(t)}{h}\right]
\left( \frac{1}{h}\int_{t-h}^{t}\frac{ds}{\left(  \inf_{s\leq z\leq
t}I(z)\right)
^{\frac{q}{q-1}}}\right)  ^{\frac{1-q}{q}}\leq\left(  \frac{1}{h}%
\int_{\{f_{\mu}^{\ast}(t)<\left|  f\right|
<f_{\mu}^{\ast}(t-h)\}}\left| \nabla f(x)\right|  ^{q}d\mu(x)\right)
^{1/q}.
\]
Letting $h\rightarrow0$ we find
\[
\left(  -f_{\mu}^{\ast}\right)  ^{\prime}(t)I(t)\leq\left(  \frac{d}{dt}%
\int_{\{\left|  f\right|  >f_{\mu}^{\ast}(t)\}}\left|  \nabla
f(x)\right| ^{q}d\mu(x)\right)  ^{1/q}.
\]
Consider a finite family of intervals $\left(  a_{i},b_{i}\right)
,$
$i=1,\ldots,m$, with $0<a_{1}<b_{1}\leq a_{2}<b_{2}\leq\cdots\leq a_{m}%
<b_{m}<1.$ The previous inequality then yields
\begin{equation*}\begin{split}
\int_{\cup_{1\leq i\leq m}(a_{i},b_{i})}\left(  \left(
-f_{\mu}^{\ast }\right)  ^{\prime}(s)I(s)\right)  ^{q}ds  &
\leq\int_{\cup_{1\leq i\leq m}(a_{i},b_{i})}\left(
\frac{d}{ds}\int_{\{\left|  f\right|  >f_{\mu}^{\ast
}(s)\}}\left|  \nabla f(x)\right|  ^{q}d\mu(x)\right)  ds\\
&  =\sum_{i=1}^{m}\int_{\left\{  f_{\mu}^{\ast}(b_{i})<\left|
f\right|  \leq
f_{\mu}^{\ast}(a_{i})\right\}  }\left|  \nabla f(x)\right|  ^{q}d\mu(x)\\
&  =\sum_{i=1}^{m}\int_{\left\{  f_{\mu}^{\ast}(b_{i})<\left|
f\right| <f_{\mu}^{\ast}(a_{i})\right\}  }\left|  \nabla f(x)\right|
^{q}d\mu(x)\text{
\ \ (by Condition 2)}\\
&  =\int_{\cup_{1\leq i\leq m}\left\{  f_{\mu}^{\ast}(b_{i})<\left|
f\right|
<f_{\mu}^{\ast}(a_{i})\right\}  }\left|  \nabla f(x)\right|  ^{q}d\mu(x)\\
&  \leq\int_{0}^{\sum_{i=1}^{m}\left(  b_{i}-a_{i}\right)  }\left(
\left| \nabla f\right|  _{\mu}^{\ast}(s)\right)  ^{q}ds.
\end{split}\end{equation*}
Now by a routine limiting process it follows that for any measurable
set $E\subset$ $(0,1)$ we have
\[
\int_{E}\left(  \left(  -f_{\mu}^{\ast}\right)
^{\prime}(s)I(s)\right) ^{q}ds\leq\int_{0}^{|E|}\left(\left|  \nabla
f\right|  _{\mu}^{\ast}(s)\right)^{q}ds.
\]
Consequently (\ref{aa}) follows from (\ref{hp}).

Case $q=1.$ Using the same procedure we arrive at
\[
\lim_{h\rightarrow0}cap_{1}\left(  t-h,\text{ }t\right)
\frac{[f_{\mu}^{\ast
}(t-h)-f_{\mu}^{\ast}(t)]}{h}\leq\lim_{h\rightarrow0}\frac{1}{h}\left(
\int_{\{f_{\mu}^{\ast}(t)<\left|  f\right|
<f_{\mu}^{\ast}(t-h)\}}\left| \nabla f(x)\right|  d\mu(x)\right)
\]
which combined with
\[
cap_{1}\left(  t-h,\text{ }t\right)  \geq\inf_{t-h\leq z\leq t}I(z)
\]
yields
\[
\left(  -f_{\mu}^{\ast}\right)  ^{\prime}(t)I(t)\leq\frac{d}{dt}%
\int_{\{\left|  f\right|  >f_{\mu}^{\ast}(t)\}}\left|  \nabla
f(x)\right| d\mu(x),
\]
as desired.

Finally to prove (\ref{reod00}) we write
\begin{equation}
f_{\mu}^{\ast}(s)-f_{\mu}^{\ast}(t)=\int_{s}^{t}\left(
-f_{\mu}^{\ast
}\right)  ^{\prime}(x)dx. \label{aa1}%
\end{equation}
Since $f\in Lip(\Omega)\cap L^{1}\left(  \Omega\right)  ,$
$f_{\mu}^{\ast\ast }(t)$ is finite for all $0<t\leq1.$ Consequently,
by (\ref{aa1}) and Fubini's theorem, we get
\begin{equation*}\begin{split}
f_{\mu}^{\ast\ast}(t)-f_{\mu}^{\ast}(t)  &
=\frac{1}{t}\int_{0}^{t}\left(
f_{\mu}^{\ast}(s)-f_{\mu}^{\ast}(t)\right)
ds=\frac{1}{t}\int_{0}^{t}\left(
\int_{s}^{t}\left(  -f_{\mu}^{\ast}\right)  ^{\prime}(x)dx\right)  ds\\
&  =\frac{1}{t}\int_{0}^{t}s\left(  -f_{\mu}^{\ast}\right)
^{\prime}(s)ds.
\end{split}\end{equation*}
By H\"{o}lder's inequality and (\ref{aa}),
\begin{equation*}\begin{split}
\int_{0}^{t}s\left(  -f_{\mu}^{\ast}\right)  ^{\prime}(s)ds  &
\leq\left( \int_{0}^{t}\left(  \left(  -f_{\mu}^{\ast}\right)
^{\prime}(s)I(s)\right)
^{q}ds\right)  ^{1/q}\frac{1}{w_{q}(t)}\\
&  \leq\left(  \int_{0}^{t}\left[  \left(  \left(
-f_{\mu}^{\ast}\right)
^{\prime}(\cdot)I(\cdot)\right)  ^{\ast}(s)\right]  ^{q}ds\right)  ^{1/q}%
\frac{1}{w_{q}(t)}\\
&  \leq\left(  \int_{0}^{t}\left(  \left|  \nabla f\right|
_{\mu}^{\ast }\right)  ^{q}(s)ds\right)  ^{1/q}\frac{1}{w_{q}(t)},
\end{split}\end{equation*}
and (\ref{reod00}) follows.
\end{proof}

\subsection{A version of Theorem \ref{capa01} without assuming Condition 2}

\begin{theorem}
\label{capa03}Let $\left(  \Omega,d,\mu\right)  $ be a metric
probability space satisfying Condition 1', and let $1\leq q<\infty.$
Then for $f\in
Lip(\Omega)\cap L^{1}\left(  \Omega\right)  ,$ we have%
\begin{equation}
(f_{\mu}^{\ast\ast}(t)-f_{\mu}^{\ast}(t))w_{q}(t)\leq\left(  \frac{1}{t}%
\int_{0}^{t}\left(  \left|  \nabla f\right|  _{\mu}^{\ast}\right)
^{q}(s)ds\right)  ^{1/q},\text{ for }t\in(0,1). \label{reor001}%
\end{equation}
\end{theorem}

\begin{proof}
We rely heavily on an argument by Emanuel Milman \cite[Remark
3.3]{Emil01} adapted to our setting. Let $\Psi$ be the class of
positive Lipschitz functions defined on $\Omega$ that, moreover,
satisfy

1) $0\leq\Phi\leq1.$

2) For every $0\leq t\leq1,$%
\begin{equation}
\int_{\left\{  \Phi=t\right\}  }\left|  \nabla\Phi\right|
^{q}d\mu=0.
\label{trunca}%
\end{equation}
Given $\Phi\in\Psi,$ the truncation$\ N[\Phi_{t_{1}}^{t_{2}}]$ (cf.
(\ref{definidanueva}) above)$,$ satisfies
\[
(t_{2}-t_{1})\left|  \nabla\left(  N[\Psi_{t_{1}}^{t_{2}}(x)]\right)
\right| \leq\left|  \nabla\Phi(x)\right|  \text{ }\ \ \text{for all
}x\in\Omega.
\]
Moreover, since $N[\Psi_{t_{1}}^{t_{2}}(x)]$ is constant on the open
sets $\left\{  \Phi>t_{2}\right\}  $ and $\left\{
\Phi<t_{1}\right\}  $, we have $\left|  \nabla\left(
N[\Phi_{t_{1}}^{t_{2}}]\right)  \right|  =0$ on these
sets, and%
\begin{equation*}\begin{split}
\int_{\Omega}\left|  \nabla\left(  N[\Phi_{t_{1}}^{t_{2}}(x)]\right)
\right| d\mu(x)  &  =\frac{1}{(t_{2}-t_{1})}\int_{\left\{  t_{1}\leq
f\leq
t_{2}\right\}  }\left|  \nabla\Phi(x)\right|  d\mu(x)\\
&  =\frac{1}{(t_{2}-t_{1})}\int_{\left\{  t_{1}<f<t_{2}\right\}
}\left|
\nabla\Phi(x)\right|  d\mu(x)\text{ \ \ \ (by (\ref{trunca})).}%
\end{split}\end{equation*}
Proceeding as in the proof of Theorem \ref{capa01}, we obtain
\[
(\Phi_{\mu}^{\ast\ast}(t)-\Phi_{\mu}^{\ast}(t))w_{q}(t)\leq\left(  \frac{1}%
{t}\int_{0}^{t}\left(  \left|  \nabla\Phi\right|
_{\mu}^{\ast}\right) ^{q}(s)ds\right)  ^{1/q}.
\]
We shall now consider two cases:

\textbf{Case 1}: Suppose that $f\in Lip(\Omega)$, $f\geq0$ and $f$
is bounded. Then, without loss of generality, we may assume
(dividing by a constant if it were necessary) that $\left\|
f\right\|  _{\infty}\leq1.$ It follows from \cite[Remark
3.3]{Emil01} that given $\varepsilon>0$ there exists
$f_{\varepsilon}\in$ $\Psi$ such that%
\[
\left\|  |\nabla f_{\varepsilon}|\right\|
_{L^{q}}\leq(1+\varepsilon)\left\| |\nabla f|\right\|  _{L^{q}}.
\]
Moreover, if we denote
\[
\Gamma=\{t\in\left[  0,1\right]  :\mu\left\{  \left\{  f=t\right\}
>0\right\}  ,
\]
the discrete countable set of atoms of $f$ under $\mu,$ then
\begin{equation}
\int_{\Gamma}\left|  \nabla f_{\varepsilon}\right|  ^{q}d\mu(x)=0.
\label{eureka}%
\end{equation}
Furthermore, let us write $\Gamma=\{\gamma_{i}\}_{i=0,1,2\cdots}$,
with $\gamma_{i}<\gamma_{i+1},$ and $G_{i}=\left\{
x:\gamma_{i}<f(x)<\gamma _{i+1}\right\}  ,$\newline $i=0,1..$ A
perusal of the construction used by E. Milman, shows that on each
$G_{i},i=0,1..,$ we have
\begin{equation}
\left|  \nabla f_{\varepsilon}(x)\right|  \leq(1+\varepsilon)\left|
\nabla
f(x)\right|  \text{.} \label{cota01}%
\end{equation}
Moreover, if we let $\varepsilon=1/n,$ then
\begin{equation}
f_{n}\underset{n\rightarrow\infty}{\rightarrow}f\text{ in }L^{1}.
\label{converge}%
\end{equation}

Since $f_{n}\in\Psi,$ the truncation argument of Theorem
\ref{capa01} works, and we find
\[
(\left(  f_{n}\right)  _{\mu}^{\ast\ast}(t)-\left(  f_{n}\right)  _{\mu}%
^{\ast}(t))\leq\left(  \frac{1}{t}\int_{0}^{t}\left(
\frac{s}{I(s)}\right)
^{\frac{q}{q-1}}ds\right)  ^{\frac{q-1}{q}}\left(  \frac{1}{t}\int_{0}%
^{t}\left(  \left|  \nabla f_{n}\right|  _{\mu}^{\ast}\right)  ^{q}%
(s)ds\right)  ^{1/q}.
\]
Therefore, for each $n\in N$ and for any Borel set $E\subset\Omega$
with $\mu(E)\leq t,$ we have
\begin{equation*}\begin{split}
\int_{E}\left|  \nabla f_{n}\right|  ^{q}d\mu(x)  &
=\int_{\Gamma\cap E}\left|  \nabla f_{n}\right|
^{q}d\mu(x)+\int_{E\setminus\Gamma}\left|
\nabla f_{n}\right|  ^{q}d\mu(x)\\
&  =\int_{E\setminus\Gamma}\left|  \nabla f_{n}\right|
^{q}d\mu(x)\text{
\ \ \ \ (by \ref{eureka})}\\
&  =\sum_{i}\int_{G_{i}}\left|  \nabla f_{n}\right|  ^{q}d\mu(x)\\
&  \leq\left(  1+\frac{1}{n}\right)
^{q}\int_{E\setminus\Gamma}\left|  \nabla
f\right|  ^{q}d\mu(x)\text{ \ (by (\ref{cota01}))}\\
&  \leq\left(  1+\frac{1}{n}\right)  ^{q}\int_{E}\left|  \nabla
f\right| ^{q}d\mu(x).
\end{split}\end{equation*}
Consequently, by (\ref{hp}), we obtain
\begin{equation*}\begin{split}
\int_{0}^{t}\left(  \left|  \nabla f_{n}\right|
_{\mu}^{\ast}\right) ^{q}(s)ds  &  =\sup_{\mu(E)\leq
t}\int_{E}\left|  \nabla f_{n}\right|
^{q}d\mu(x)\\
&  \leq\left(  1+\frac{1}{n}\right)  ^{q}\sup_{\mu(E)\leq
t}\int_{E}\left|
\nabla f\right|  ^{q}d\mu(x)\\
&  =\left(  1+\frac{1}{n}\right)  ^{q}\int_{0}^{t}\left(  \left|
\nabla f\right|  _{\mu}^{\ast}\right)  ^{q}(s)ds.
\end{split}\end{equation*}
On the other hand from (\ref{converge}) we get (cf. \cite[Lemma
2.1]{gar}):
\begin{equation*}\begin{split}
\left(  f_{n}\right)  _{\mu}^{\ast\ast}(t)  &  \rightarrow
f_{\mu}^{\ast\ast
}(t),\text{ uniformly for }t\in\lbrack0,1]\text{, and }\\
\left(  f_{n}\right)  _{\mu}^{\ast}(t)\rightarrow
f_{\mu}^{\ast}(t)\text{ } &  \text{at all points of continuity of
}f_{\mu}^{\ast}.
\end{split}\end{equation*}
Thus, letting $n\rightarrow\infty$ we obtain (\ref{reor001}).

\textbf{Case 2}: Suppose that $f$ is a positive Lip function.
Consider an
increasing sequence of positive number $a_{n}$ such that $\lim_{n}a_{n}%
=\infty,$ and such that, moreover, the sets $D_{n}=\{x:f(x)=a_{n}\}$
have $\mu-$measure $0,$ for all $n.$ Let
\[
h_{n}=\left\{
\begin{array}
[c]{cc}%
a_{n} & \text{if }f(x)\geq a_{n},\\
f(x) & \text{if }f(x)<a_{n}.
\end{array}
\right.
\]
Apply the result obtained in the first part of the proof to each of
the
$h_{n}^{\prime}s.$ We obtain%
\begin{equation}
(\left(  h_{n}\right)  _{\mu}^{\ast\ast}(t)-\left(  h_{n}\right)  _{\mu}%
^{\ast}(t))\leq\left(  \frac{1}{t}\int_{0}^{t}\left(
\frac{s}{I(s)}\right)
^{\frac{q}{q-1}}ds\right)  ^{\frac{q-1}{q}}\left(  \frac{1}{t}\int_{0}%
^{t}\left(  \left|  \nabla h_{n}\right|  _{\mu}^{\ast}\right)  ^{q}%
(s)ds\right)  ^{1/q}. \label{ladoizqa}%
\end{equation}
Since for each $n\in N$ the set $A_{n}=\{x:f(x)<a_{n}\}$ is open, we
have \ $\left|  \nabla h_{n}(x)\right|  =\left|  \nabla f(x)\right|
,$ a.e. $x\in A_{n}.$

Given a measurable set $E\subset\Omega,$ with $\mu(E)\leq t,$
\begin{equation*}\begin{split}
\int_{E}\left|  \nabla h_{n}\right|  ^{q}d\mu &  =\int_{E\cap
A_{n}}\left| \nabla h_{n}\right|  ^{q}d\mu+\int_{E\setminus
A_{n}}\left|  \nabla
h_{n}\right|  ^{q}d\mu\\
&  =\int_{\Gamma\cap A_{n}}\left|  \nabla h_{n}\right|
^{q}d\mu\text{
\ \ (since }\mu(D_{n})=0)\\
&  =\int_{E\cap A_{n}}\left|  \nabla f\right|  ^{q}d\mu\\
&  \leq\int_{E}\left|  \nabla f\right|  ^{q}d\mu.
\end{split}\end{equation*}
Thus
\[
\int_{0}^{t}\left(  \left|  \nabla h_{n}\right|
_{\mu}^{\ast}\right) ^{q}(s)ds\leq\int_{0}^{t}\left(  \left|  \nabla
f\right|  _{\mu}^{\ast }\right)  ^{q}(s)ds.
\]
To take care of the left hand side of (\ref{ladoizqa}) we can use
again \cite[Lemma 2.1]{gar} noting that, by monotone convergence,
\[
h_{n}\underset{n\rightarrow\infty}{\rightarrow}f\text{ in }L^{1}.
\]
Combining our findings we can conclude the proof of (\ref{reor001}).
\end{proof}

\subsection{Poincar\'{e} inequalities on r.i. spaces\label{sec24}}

From now on we will assume that our metric probability spaces
$\left( \Omega,d,\mu\right)  $ satisfy Condition 1'.

We consider Banach function spaces on $({\Omega},d,\mu)$ with the
property if $g\in X$ and $f$ is a $\mu-$measurable function on
${\Omega}$ such that $f_{\mu}^{\ast}=g_{\mu}^{\ast},$ then $f\in X,$
and, moreover, $\Vert f\Vert_{X}=\Vert g\Vert_{X}$. We say that
$X=X({\Omega})$ is a rearrangement-invariant (r.i.)
space\footnote{We refer the reader to \cite{bs}
for a detailed treatment.}. It follows that%

\begin{equation}
L^{\infty}(\Omega)\subset X(\Omega)\subset L^{1}(\Omega), \label{nuevadeli}%
\end{equation}
with continuous embeddings.

An r.i. space $X=X({\Omega})$ can be represented by a r.i. space $\bar{X}%
=\bar{X}(0,1)$ on the interval $(0,1),$ with Lebesgue
measure\footnote{A characterization of the norm
$\Vert\cdot\Vert_{\bar{X}}$ is available (see \cite[Theorem 4.10 and
subsequent remarks]{bs})} in the sense that for $f\in X,$
\[
\Vert f\Vert_{X}=\Vert f_{\mu}^{\ast}\Vert_{\bar{X}}.
\]
Let us also record here the Hardy-Calder\'{o}n property
\begin{equation}
f_{\mu}^{\ast\ast}\leq g_{\mu}^{\ast\ast}\Rightarrow\left\|
f\right\|
_{X}\leq\left\|  g\right\|  _{X}. \label{hardy}%
\end{equation}

Typical examples of r.i. spaces are the $L^{p}$-spaces, Lorentz
spaces and Orlicz spaces.

The Boyd indices, $\bar{\alpha}_{X},\underline{\alpha}_{X},$ of a
r.i. space $X$ (cf. \cite{bs} for details) are defined by
\[
\bar{\alpha}_{X}=\inf\limits_{s>1}\dfrac{\ln h_{X}(s)}{\ln s}\text{
\ \ and \ \ }\underline{\alpha}_{X}=\sup\limits_{s<1}\dfrac{\ln
h_{X}(s)}{\ln s},
\]
where $h_{X}(s)$ denotes the norm of the dilation operator $E_{s},$
$s>0,$ on $\bar{X}$, defined by\footnote{The operator $E_{s}$ is
bounded on $\bar{X}$
for every r.i. space $X(\Omega)$ and for every $s>0$. Moreover,%
\[
h_{X}(s)\leq\max\{1,s\}.
\]
$.$}
\[
E_{s}f(t)=\left\{
\begin{array}
[c]{ll}%
f^{\ast}(\frac{t}{s}) & 0<t<s,\\
0 & s<t<1.
\end{array}
\right.
\]
For example, if $X=L^{p}$, then
$\overline{\alpha}_{L^{p}}=\underline{\alpha }_{L^{p}}=\frac{1}{p}.$
Consider the Hardy operators defined by
\[
Pf(t)=\frac{1}{t}\int_{0}^{t}f(s)ds;\text{ \ \ \
}Qf(t)=\int_{t}^{\infty }f(s)\frac{ds}{s}.
\]
It is well known that (cf. \cite{bs})
\begin{equation}%
\begin{array}
[c]{c}%
P\text{ is bounded on }\bar{X}\text{ }\Leftrightarrow\overline{\alpha}%
_{X}<1,\\
Q_{a}\text{ is bounded on }\bar{X}\text{
}\Leftrightarrow\underline{\alpha }_{X}>a.
\end{array}
\label{alcance}%
\end{equation}

The $q-$\textbf{convexification} $X^{(q)}$ of a r.i. space $X$ are
defined by the condition%
\begin{equation}\label{scalaq}
X^{(q)}=\{f:\left|  f\right|  ^{q}\in X\},\text{ \ \ }\left\|
f\right\| _{X^{(q)}}=\left\|  \left|  f\right|  ^{q}\right\|
_{X}^{1/q}. \end{equation}

The $q-$\textbf{capacitary spaces} $LS_{q}(X)$ associated with a
r.i. space $X$ are
defined by the condition%
\[
\left\|  f\right\|  _{LS_{q}(X)}:=\left\|  \left(  f_{\mu}^{\ast\ast
}(t)-f_{\mu}^{\ast}(t)\right)  w_{q}(t)\right\|  _{\bar{X}}<\infty.
\]
It follows from (\ref{orden}) that these functionals increase with
the parameter $q.$

The $q-$\textbf{capacitary Hardy operator} $Q_{w_{q}}$ is defined on
positive measurable positive functions on $(0,1)$ by
\[
Q_{w_{q}}f(t)=\int_{t}^{1}f(s)\frac{ds}{sw_{q}(s)}.
\]

\subsubsection{Poincar\'{e} inequalities and the capacitary Hardy operator
$Q_{w_{q}}$}

\begin{theorem}
\label{capa02}Let $X(\Omega),Y(\Omega)$ be r.i. spaces such that
$\bar{\alpha }_{X}<1.$ Let $q\geq1,$ and suppose that there exists a
constant $c=c(q)$ such that for every positive function
$f\in\bar{X}^{(q)},$ with supp$f\subset
(0,1/2),$ we have%
\begin{equation}
\left\|  Q_{w_{q}}f\right\|  _{\bar{Y}^{(q)}}\leq c\left\|
f\right\|
_{\bar{X}^{(q)}}. \label{propiedad}%
\end{equation}
Then there exists a constant $C=C(c,\left\|  P\right\|  _{\bar{X}%
\rightarrow\bar{X}})$ such that for all $g\in Lip(\Omega)\cap
L^{1}\left(
\Omega\right)  ,$%
\begin{equation}
\left\|  g-\int_{\Omega}gd\mu\right\|  _{Y^{(q)}}\leq C\left(
\left\|
\left|  \nabla g\right|  \right\|  _{X^{(q)}}+\left\|  g-\int_{\Omega}%
gd\mu\right\|  _{L^{1}}\right)  . \label{extra}%
\end{equation}
Moreover,
\begin{align}
\left\|  g-\int_{\Omega}gd\mu\right\|  _{Y^{(q)}}  &  \leq C\left\|
g-\int_{\Omega}gd\mu\right\|  _{LS_{q}(X^{(q)})}\nonumber\\
&  \leq C\left(  \left\|  \left|  \nabla g\right|  \right\|  _{X^{(q)}%
}+\left\|  g-\int_{\Omega}gd\mu\right\|  _{L^{1}}\right)  . \label{perdida02}%
\end{align}
\end{theorem}

\begin{proof}
Since $g\in Lip(\Omega)\cap L^{1}\left(  \Omega\right)  ,$
$g_{\mu}^{\ast\ast }(t)$ is finite for all $0<t\leq1,$ thus
\begin{equation*}\begin{split}
g_{\mu}^{\ast}(t)^{q}  &  \leq g_{\mu}^{\ast\ast}(t)^{q}=\left(
\int _{t}^{1/2}\left(  -g_{\mu}^{\ast\ast}\right)
^{\prime}(s)ds+g_{\mu}^{\ast
\ast}(1/2)\right)  ^{q}\\
&  =\left(  \int_{t}^{1/2}\left(
g_{\mu}^{\ast\ast}(s)-g_{\mu}^{\ast
}(s)\right)  \frac{ds}{s}+g_{\mu}^{\ast\ast}(1/2)\right)  ^{q}\\
&  =\left(  \int_{t}^{1/2}\left(
g_{\mu}^{\ast\ast}(s)-g_{\mu}^{\ast }(s)\right)
w_{q}(s)\frac{ds}{w_{q}(s)s}+g_{\mu}^{\ast\ast}(1/2)\right) ^{q}.
\end{split}\end{equation*}
Consequently,
\begin{equation*}\begin{split}
\left\|  g\right\|  _{Y^{(q)}}  &  =\left\|  \left(
g_{\mu}^{\ast}\right)
^{q}\right\|  _{\bar{Y}}^{1/q}\\
&  \leq\left\|  \left(  \int_{t}^{1/2}\left(
g_{\mu}^{\ast\ast}(s)-g_{\mu }^{\ast}(s)\right)
w_{q}(s)\frac{ds}{w_{q}(s)s}+g_{\mu}^{\ast\ast
}(1/2)\right)  ^{q}\right\|  _{\bar{Y}}^{1/q}\\
&  =\left\|  \int_{t}^{1/2}\left(
g_{\mu}^{\ast\ast}(s)-g_{\mu}^{\ast }(s)\right)
w(s)\frac{ds}{w_{q}(s)s}+g_{\mu}^{\ast\ast}(1/2)\right\|
_{\bar{Y}^{(q)}}\\
&  \preceq\left\|  \int_{t}^{1/2}\left(
g_{\mu}^{\ast\ast}(s)-g_{\mu}^{\ast }(s)\right)
w_{q}(s)\frac{ds}{w_{q}(s)s}\right\|  _{\bar{Y}^{(q)}}+\left\|
g\right\|  _{L^{1}}\\
&  \preceq\left\|  \left(
g_{\mu}^{\ast\ast}(s)-g_{\mu}^{\ast}(s)\right)
w_{q}(s)\right\|  _{\bar{X}^{(q)}}+\left\|  g\right\|  _{L^{1}}\\
&  =\left\|  \left[  \left(
g_{\mu}^{\ast\ast}(s)-g_{\mu}^{\ast}(s)\right)
w_{q}(s)\right]  ^{q}\right\|  _{\bar{X}}^{1/q}+\left\|  g\right\|  _{L^{1}}\\
&  \preceq\left\|  \frac{1}{t}\int_{0}^{t}\left(  \left|  \nabla
g\right| _{\mu}^{\ast}\right)  ^{q}(s)ds\right\|
_{\bar{X}}^{1/q}+\left\|  g\right\|
_{L^{1}}\text{ \ \ \ \ (by (\ref{reod00}))}\\
&  \preceq\left\|  \left(  \left|  \nabla g\right|
_{\mu}^{\ast}\right) ^{q}\right\|  _{\bar{X}}^{1/q}+\left\|
g\right\|  _{L^{1}}\text{
\ \ \ \ (since }\bar{\alpha}_{X}<1)\\
&  =\left\|  \left|  \nabla g\right|  _{\mu}^{\ast}\right\|  _{\bar{X}^{(q)}%
}+\left\|  g\right\|  _{L^{1}}\\
&  =\left\|  \left|  \nabla g\right|  \right\|  _{X^{(q)}}+\left\|
g\right\| _{L^{1}}.
\end{split}\end{equation*}
Therefore,%
\[
\left\|  g-\int_{\Omega}gd\mu\right\|  _{Y^{(q)}}\preceq\left\|
\left|
\nabla g\right|  \right\|  _{X^{(q)}}\text{ }+\left\|  g-\int_{\Omega}%
gd\mu\right\|  _{L^{1}}.
\]
Notice that (\ref{perdida02}) is implicit in the proof.
\end{proof}

\begin{remark}
If Cheeger's inequality holds for $({\Omega},d,\mu)$, i.e. if there
exists $C_{e}$ such that
\[
\left\|  g-\int_{\Omega}gd\mu\right\|  _{L^{1}}\leq C_{e}\left\|
\left|
\nabla g\right|  \right\|  _{L^{1}}\text{,}%
\]
then the extra $L^{1}$ term that appears in (\ref{extra}) and (\ref{perdida02}%
) can be left out. Indeed, combining Cheeger's inequality with
(\ref{nuevadeli}) yields%
\[
\left\|  g-\int_{\Omega}gd\mu\right\|  _{L^{1}}\leq C_{e}\left\|
\left| \nabla g\right|  \right\|  _{L^{1}}\leq C_{e}\bar{c}\left\|
\left|  \nabla g\right|  \right\|  _{X^{(q)}},
\]
where $\bar{c}$ denotes the embedding constant of $X^{(q)}\subset
L^{1}$ (cf. (\ref{nuevadeli})).
\end{remark}

\begin{remark}
For $q=1$, the condition (\ref{propiedad}) reads: there exists a
constant $C$ such that for every positive function $f\in\bar{X}$
with supp$f\subset (0,1/2),$ we have
\[
\left\|  \int_{t}^{1}f(s)\frac{ds}{s\left(
\inf_{0<z<t}\frac{I(z)}{z}\right) }\right\|  _{\bar{Y}}\leq C\left\|
f\right\|  _{\bar{X}}.
\]
\end{remark}

\subsubsection{Poincar\'{e} inequalities: a limiting case}

The limiting case where $\bar{\alpha}_{X}=1$ is not covered by
Theorem
\ref{capa02}. For example, if $X=L^{1}$ then the condition $($\ref{perdida02}%
$)$ reads as
\[
\left\|  g-\int_{\Omega}gd\mu\right\|  _{L^{q}}\preceq\left\|
g-\int_{\Omega }gd\mu\right\|  _{LS(L^{q})}\preceq\left\|  \left|
\nabla g\right|  \right\| _{L^{q}},
\]
but unfortunately we cannot apply the theorem since
$\bar{\alpha}_{L^{1}}=1$.

In this section we formulate conditions for the validity of
Poincar\'{e} inequalities in terms of $\nu(s)=cap_{1}(s,1/2)/s.$ Let
$q\in\lbrack 1,\infty),$ then we say that $\nu$ is a $q-$Muckenhoupt
weight iff there exists a constant $c>0$ such that
\begin{equation}
\left\|  Pf\right\|  _{L^{q}(\nu)}\leq c\left\|  f\right\|
_{L^{q}(\nu)}.
\label{muk}%
\end{equation}

Using the usual description of Muckenhoupt weights (cf. \cite{mu1},
\cite{mu2}) we see that $\nu(s)=cap_{1}(s,1/2)/s$ is a
$q-$Muckenhoupt weight iff there exists a constant $c>0$ such that
for all $0<t<1/2$,
\[
\left\{
\begin{array}
[c]{cc}%
\left(  \int_{t}^{1/2}\left(  \frac{cap_{1}(s,1/2)}{s}\right)
^{q}\frac
{ds}{s^{q}}\right)  ^{1/q}\left(  \int_{0}^{t}\left(  \frac{s}{cap_{1}%
(s,1/2)}\right)  ^{\frac{q}{q-1}}ds\right)  ^{\frac{^{q-1}}{q}}\leq
c &
\text{if }1<q,\\
\int_{t}^{1/2}\frac{cap_{1}(s,1/2)}{s}\frac{ds}{s}\leq c\frac{cap_{1}%
(t,1/2)}{t} & \text{if }q=1.
\end{array}
\right.
\]

We now show that if $\nu$ is $q-$Muckenhoupt weight then
Poincar\'{e} inequalities can be described in terms of the Hardy
isoperimetric operator
\[
Q_{cap_{1}}f(t)=\int_{t}^{1/2}f(s)\frac{ds}{cap_{1}(s,1/2)}.
\]

\begin{theorem}
Let $q\geq1,$ and suppose that $\frac{cap_{1}(t,1/2)}{t}$ is a
$q-$Muckenhoupt weight. Then, there exists a constant $c>0$ such
that for all\ $f\in
Lip(\Omega)\cap L^{1}(\Omega),$ $0<t<1/2,$ we have%
\begin{equation}
\int_{0}^{t}\left[  \left(
f_{\mu}^{\ast\ast}(s)-f_{\mu}^{\ast}(s)\right)
\frac{cap_{1}(s,1/2)}{s}\right]  ^{q}ds\leq c\int_{0}^{t}\left(
\left|
\nabla f\right|  _{\mu}^{\ast}\right)  ^{q}(s)ds. \label{fii}%
\end{equation}
In particular, if $X$ is a r.i. space with
$\underline{\alpha}_{X}>0,$ there exists an absolute constant $C$
(depending on $c,$ the $q-$Muckenhoupt norm of
$\frac{cap_{1}(t,1/2)}{t},$ and the norms of the Hardy operators
$P,Q)$ such that for all $f\in Lip(\Omega)\cap L^{1}(\Omega),$
\begin{equation}
\left\|  \lbrack f_{\mu}^{\ast\ast}(t)-f_{\mu}^{\ast}(t)]\frac{cap_{1}%
(t,1/2)}{t}\right\|  _{\bar{X}^{(q)}}\leq C\,\left(  \left\|  \left|
\nabla
f\right|  \right\|  _{X^{(q)}}+\left\|  f-\int_{\Omega}fd\mu\right\|  _{L^{1}%
}\right)  . \label{fi}%
\end{equation}
Moreover, suppose that there exists $\tilde{C}>0$ such that for
every positive function $f\in\bar{X}^{(q)},$ we have
\[
\left\|  Q_{cap_{1}}f\right\|  _{\bar{Y}^{(q)}}\leq\tilde{C}\left\|
f\right\|  _{\bar{X}^{(q)}}.
\]
Then, there exist absolute constants $C_{1},C_{2}$ (that depend on
all the previous constants as well as $\tilde{C})$ such that for all
$g\in
Lip(\Omega)\cap L^{1}(\Omega),$%
\begin{equation}
\begin{split}
\left\|  g-\int_{\Omega}gd\mu\right\|  _{Y^{(q)}}  &  \leq
C_{1}\left\| \left[  \left(  g-\int_{\Omega}gd\mu\right)
_{\mu}^{\ast\ast}(t)-\left(
g-\int_{\Omega}gd\mu\right)  _{\mu}^{\ast}(t)\right]  \frac{cap_{1}(t,1/2)}%
{t}\right\|  _{\bar{X}^{(q)}}\label{porfin}\\
&  \leq C_{2}\left(  \left\|  \left|  \nabla g\right|  \right\|  _{X^{(q)}%
}+\left\|  g-\int_{\Omega}gd\mu\right\|  _{L^{1}}\right)  .\nonumber
\end{split}
\end{equation}
\end{theorem}

\begin{proof}
Suppose that $0\leq f\in Lip(\Omega)\cap L^{1}$ satisfies\textbf{\ }
\begin{equation}
\text{For every }c\in\mathbb{R},\text{we have that }|\nabla
f(x)|=0,\text{
}\mu-\text{ a.e. on }\{x:f(x)=c\}. \label{condition2}%
\end{equation}
Then by the proof of Theorem \ref{capa01}, $f_{\mu}^{\ast}$ is
locally absolutely continuous, and
\[
f_{\mu}^{\ast\ast}(t)-f_{\mu}^{\ast}(t)=\frac{1}{t}\int_{0}^{t}s\left(
-f_{\mu}^{\ast}\right)  ^{\prime}(s)ds.
\]
Thus,
\begin{equation*}\begin{split}
\int_{0}^{t}\left[  \left(
f_{\mu}^{\ast\ast}(s)-f_{\mu}^{\ast}(s)\right)
\frac{cap_{1}(s,1/2)}{s}\right]  ^{q}ds  &  =\int_{0}^{t}\left[
\left( \frac{1}{s}\int_{0}^{s}z\left(  -f_{\mu}^{\ast}\right)
^{\prime}(z)dz\right)
\frac{cap_{1}(s,1/2)}{s}\right]  ^{q}ds\\
&  \leq c\int_{0}^{t}\left[  \left(  \left(  -f_{\mu}^{\ast}\right)
^{\prime
}(s)\right)  cap_{1}(s,1/2)\right]  ^{q}ds\text{ \ \ \ \ \ (by (\ref{muk})}\\
&  =c\int_{0}^{t}\left[  \left(  \left(  -f_{\mu}^{\ast}\right)
^{\prime }(s)\right)  \inf_{s\leq z<1/2}I(z)\right]  ^{q}ds\text{ \
\ \ \ \ (by
(\ref{mazya2}))}\\
&  \leq c\int_{0}^{t}\left[  \left(  \left(  -f_{\mu}^{\ast}\right)
^{\prime
}(s)\right)  I(s)\right]  ^{q}ds\\
&  \leq c\int_{0}^{t}\left(  \left|  \nabla f\right|
_{\mu}^{\ast}\right)
^{q}(s)ds\text{ \ \ \ \ (by (\ref{aa}). }%
\end{split}\end{equation*}
Now, using a familiar limiting argument we can avoid the restriction
(\ref{condition2}) and still achieve (\ref{fi}).

Applying the operator $Q_{1/2}f(s)=$
$\int_{s}^{1/2}f(z)\frac{dz}{z}$ to the inequality (\ref{fii}), and
then combining with the fact that for $0<t<1/2$ we have (see
\cite{bmrlibro})
\begin{align}
P(Q_{1/2}f)(s)-2\int_{0}^{1/2}f  &  =Q_{1/2}(Pf)(s),\text{ }\nonumber\\
Pf(s)+Q_{1/2}f(s)  &  =P(Q_{1/2}f)(s), \label{propi}%
\end{align}
we obtain
\begin{equation*}\begin{split}
&  \frac{1}{t}\int_{0}^{t}Q_{1/2}\left[  \left(  f_{\mu}^{\ast\ast}%
(\cdot)-f_{\mu}^{\ast}(\cdot)\right)
\frac{cap(\cdot,1/2)}{(\cdot)}\right]
^{q}\,(s)ds-2\int_{0}^{1/2}\left[  \left(  f_{\mu}^{\ast\ast}(s)-f_{\mu}%
^{\ast}(s)\right)  \frac{cap(\cdot,1/2)}{s}\right]  ^{q}ds\\
&  \leq c\left(  \frac{1}{t}\int_{0}^{t}Q_{1/2}\left(  \left(
\left|  \nabla f\right|  _{\mu}^{\ast}\right)  ^{q}(s)\right)
\,ds-2\int_{0}^{1/2}\left(
\left|  \nabla f\right|  _{\mu}^{\ast}\right)  ^{q}(s)\right) \\
&  \leq\frac{c}{t}\int_{0}^{t}Q_{1/2}\left(  \left(  \left|  \nabla
f\right| _{\mu}^{\ast}\right)  ^{q}(s)\right)  \,ds.
\end{split}\end{equation*}
Moreover, since
\begin{equation*}\begin{split}
\frac{1}{1/2}\int_{0}^{1/2}\left[  \left(  f_{\mu}^{\ast\ast}(s)-f_{\mu}%
^{\ast}(s)\right)  \frac{cap_{1}(s,1/2)}{s}\right]  ^{q}ds  &  \leq\frac{1}%
{t}\int_{0}^{t}\left[  \left(
f_{\mu}^{\ast\ast}(s)-f_{\mu}^{\ast}(s)\right)
\frac{cap_{1/2}(s,1/2)}{s}\right]  ^{q}ds\\
&  \leq\frac{c}{t}\int_{0}^{t}\left(  \left|  \nabla f\right|
_{\mu}^{\ast
}\right)  ^{q}(s)ds\text{ \ \ (by (\ref{fii}))}\\
&  \leq\frac{c}{t}\int_{0}^{t}Q_{1/2}\left(  \left(  \left|  \nabla
f\right|
_{\mu}^{\ast}\right)  ^{q}(s)\right)  \,ds\text{ \ \ (by (\ref{propi})),}%
\end{split}\end{equation*}
we obtain
\[
\int_{0}^{t}Q_{1/2}\left[  \left(
f_{\mu}^{\ast\ast}(\cdot)-f_{\mu}^{\ast
}(\cdot)\right)  \frac{cap_{1}(\cdot,1/2)}{(\cdot)}\right]  ^{q}%
(s)\,ds\leq2c\int_{0}^{t}Q_{1/2}\left(  \left(  \left|  \nabla
f\right| _{\mu}^{\ast}\right)  ^{q}\right)  (s)ds.
\]
Observe that $Q_{1/2}h(s)$ is decreasing. Consequently,
\begin{equation*}\begin{split}
\int_{0}^{t}Q_{1/2}\left[  \left(
f_{\mu}^{\ast\ast}(\cdot)-f_{\mu}^{\ast }(\cdot)\right)
\frac{cap_{1}(\cdot,1/2)}{(\cdot)}\right]  ^{q}(s)ds  &
=\int_{0}^{t}\left(  Q_{1/2}\left[  \left(
f_{\mu}^{\ast\ast}(\cdot)-f_{\mu }^{\ast}(\cdot)\right)
\frac{cap_{1}(\cdot,1/2)}{(\cdot)}\right]
^{q}\right)  ^{\ast}(s)ds\\
&  \leq2c\int_{0}^{t}Q_{1/2}\left(  \left(  \left|  \nabla f\right|
_{\mu }^{\ast}\right)  ^{q}\right)  (s)ds.
\end{split}\end{equation*}
We may now apply (\ref{hardy})
\[
\left\|  Q_{1/2}\left[  \left(  f_{\mu}^{\ast\ast}(\cdot)-f_{\mu}^{\ast}%
(\cdot)\right)  \frac{cap_{1}(\cdot,1/2)}{(\cdot)}\right]
^{q}(t)\right\| _{X}\leq C\,\left\|  Q_{1/2}\left(  \left(  \left|
\nabla f\right|  _{\mu }^{\ast}\right)  ^{q}\right)  (t)\right\|
_{X}.
\]
Whence, if $\underline{\alpha}_{X}>0,$
\[
\left\|  Q_{1/2}\left[  \left(  f_{\mu}^{\ast\ast}(\cdot)-f_{\mu}^{\ast}%
(\cdot)\right)  \frac{cap_{1}(\cdot,1/2)}{(\cdot)}\right]
^{q}(t)\right\| _{\bar{X}}\leq C\,\left\|  \left|  \nabla f\right|
^{q}\right\|  _{X}.
\]
Since both, $t\rightarrow t\left(
f_{\mu}^{\ast\ast}(t)-f_{\mu}^{\ast }(t)\right)  $ and $t\rightarrow
cap_{1}(t,1/2)$ are increasing, we see that
for $0<t<1/4$%
\begin{equation*}\begin{split}
&  \int_{t}^{1/2}\left[  \left(  f_{\mu}^{\ast\ast}(s)-f_{\mu}^{\ast
}(s)\right)  \frac{cap_{1}(s,1/2)}{s}\right]  ^{q}\frac{ds}{s}\\
&  \geq\left(  t\left(
f_{\mu}^{\ast\ast}(t)-f_{\mu}^{\ast}(t)\right) cap_{1}(t,1/2)\right)
^{q}\int_{t}^{1/2}\left[  \frac{1}{s^{2}}\right]
^{q}\frac{ds}{s}\\
&  \succeq\left(  \left(
f_{\mu}^{\ast\ast}(t)-f_{\mu}^{\ast}(t)\right)
\frac{cap_{1}(t,1/2)}{t}\right)  ^{q}.
\end{split}\end{equation*}
Using the elementary estimation
\[
\left\|  h\right\|  _{\bar{X}}\preceq\left\|
h\chi_{(0,1/4)}\right\| _{\bar{X}},
\]
we find
\begin{equation*}\begin{split}
\left\|  \left(  \lbrack
f_{\mu}^{\ast\ast}(t)-f_{\mu}^{\ast}(t)]\frac
{cap_{1}(t,1/2)}{t}\right)  ^{q}\right\|  _{\bar{X}}  &
\preceq\left\|
\left(  \lbrack f_{\mu}^{\ast\ast}(t)-f_{\mu}^{\ast}(t)]\frac{cap_{1}%
(t,1/2)}{t}\right)  ^{q}\chi_{(0,1/4)}(t)\right\|  _{\bar{X}}\\
&  \preceq\left\|  Q_{1/2}\left[  \left(
f_{\mu}^{\ast\ast}(\cdot)-f_{\mu
}^{\ast}(\cdot)\right)  \frac{cap_{1}(\cdot,1/2)}{(\cdot)}\right]  ^{q}%
(t)\chi_{(0,1/4)}(t)\right\|  _{\bar{X}}\\
&  \preceq\left\|  \left(  \left|  \nabla f\right|
_{\mu}^{\ast}\right) ^{q}\right\|  _{X}.
\end{split}\end{equation*}

Finally, to see (\ref{porfin}), we proceed as in theorem
\ref{capa02},
\[
f_{\mu}^{\ast}(t)^{q}\leq f_{\mu}^{\ast\ast}(t)^{q}=\left(  \int_{t}%
^{1/2}[f_{\mu}^{\ast\ast}(t)-f_{\mu}^{\ast}(t)]\frac{cap_{1}(s,1/2)}{s}%
\frac{ds}{cap_{1}(s,1/2)}+f_{\mu}^{\ast\ast}(1/2)\right)  ^{q}.
\]
Consequently,
\begin{equation*}\begin{split}
\left\|  f_{\mu}^{\ast}\right\|  _{Y^{(q)}}  &  =\left\|  \left(
f_{\mu
}^{\ast\ast}\right)  ^{q}\right\|  _{\bar{Y}}^{1/q}\\
&  \leq\left\|  \left(
\int_{t}^{1/2}[f_{\mu}^{\ast\ast}(t)-f_{\mu}^{\ast
}(t)]\frac{cap_{1}(s,1/2)}{s}\frac{ds}{cap_{1}(s,1/2)}+f_{\mu}^{\ast\ast
}(1/2)\right)  ^{q}\right\|  _{\bar{Y}}^{1/q}\\
&  \preceq\left\|  \left(
f_{\mu}^{\ast\ast}(s)-f_{\mu}^{\ast}(s)\right)
\frac{cap_{1}(s,1/2)}{s}\right\|  _{\bar{X}^{(q)}}+\left\|
f\right\|
_{L^{1}}\\
&  \preceq\left\|  \left(  \left|  \nabla f\right|
_{\mu}^{\ast}\right) ^{q}\right\|  _{\bar{X}}^{1/q}+\left\|
f\right\|  _{L^{1}}\text{ \ \ \ \ (by
(\ref{fi})).}%
\end{split}\end{equation*}
\end{proof}

\begin{remark}
Observe that for $0<t<1/2,$%
\[
\frac{cap_{1}(t,1/2)}{t}=\frac{\inf_{t\leq z<1/2}I(z)}{t}=\inf_{0<s<t}%
\frac{\inf_{t\leq
s<1/2}I(s)}{s}\leq\inf_{0<s<t}\frac{I(s)}{s}=w_{1}(t).
\]
Thus
\[
\left\|  \lbrack g_{\mu}^{\ast\ast}(t)-g_{\mu}^{\ast}(t)]\frac{cap_{1}%
(t,1/2)}{t}\right\|  _{\bar{X}^{(q)}}\preceq\left\|  \lbrack
g_{\mu}^{\ast \ast}(t)-g_{\mu}^{\ast}(t)]w_{1}(t)\right\|
_{\bar{X}^{(q)}}.
\]
\end{remark}

\begin{remark}
As in Theorem \ref{capa02}, the extra $L^{1}-$terms appearing in
(\ref{fi}) and (\ref{porfin}) can be omitted if Cheeger's inequality
holds. Notice that Cheeger's inequality is equivalent to (cf.
\cite{MiE} and the references therein)
\[
cap_{1}(t,1/2)\geq ct\text{ \ \ (}0<t\leq1/2),
\]
which in turn is equivalent to (cf. \cite{mmadv})
\[
\left\|  Q_{cap_{1}}f\right\|  _{L^{1}}\leq C\left\|  f\right\|
_{L^{1}},
\]
for all positive functions $f\in L^{1},$ with supp$f\subset(0,1/2).$
\end{remark}

\section{A connection with martingale inequalities via interpolation theory
and optimization}

It was shown recently in \cite{cjm} that using interpolation theory
one can relate the rearrangement inequalities for Sobolev functions
we have obtained in our work with the extrapolation theory of
martingale inequalities of Burkholder and Gundy \cite{bg} and Herz
\cite{he}. There are two key observations underlying these
developments: (i) the idea to treat truncations as part of the more
general process of decomposing elements. Here the appropriate
setting is the real method of interpolation, where decompositions
are selected using penalty methods. From this point of view our
method is related to the fact that certain optimal splittings in
interpolation theory are given by truncations; (ii) the fact that
gradients (and other related operations in analysis, e.g. the
martingale square functions!) commute, in suitable quantified
manners, with respect to these splittings.

We thought it would be worthwhile to present these ideas to this
community using a presentation that goes directly to the heart of
the matter. Thus we will focus our discussion here on Sobolev
inequalities and refer the reader to \cite{cjm} for complete proofs
and other developments. This topic will also be discussed in
\cite{mamifuture}. One reason we were originally interested in
placing our results in a larger context is that it may help to
suggest a suitable substitute for the truncation method when we deal
with higher order differential operators where truncations are
obviously inadequate.

We start by placing the truncation method within the larger context
of interpolation theory. The basic modern ingredient of *real
interpolation* is the study of controlled splittings of elements
using ``penalty'' methods (``Peetre's $K-$functional). The point of
departure of this theory are pairs $\vec{X}=(X_{0},X_{1})$ of Banach
spaces that are ``compatible'' in the sense that both spaces
$X_{i},i=0,1,$ are continuously embedded in a common Hausdorff
topological vector space\footnote{For example, this will happen if
$X_{1}\subset X_{0},$ with a continuous embedding.}. For such pairs
the sum space $\Sigma(\vec{X})=X_{0}+X_{1},$ makes sense and for
$t>0$ we can consider the parametrized family of penalty problems
given by
\begin{equation}
K(t,x;\vec{X})=\inf\left\{  \left\|  x_{0}\right\|
_{X_{0}}+t\left\| x_{1}\right\|  _{X_{1}}:x=x_{0}+x_{1},x_{i}\in
X_{i},i=0,1\right\}  .
\label{k1}%
\end{equation}
To see the effect of the penalty $t$ let us suppose, for example,
that $X_{1}\subset X_{0},$ with $\left\|  x\right\|
_{X_{0}}\leq\left\| x\right\|  _{X_{1}}.$ If $t$ is ``very large'',
say $t>1,$ then for every $x\in X_{0}$ the spliting $x=x+0$ ``wins''
and we see that $K(t,x;\vec {X})=\left\|  x\right\|  _{X_{0}},$
while on the other hand, if $x\in X_{1},$ we see that
$\lim_{t\rightarrow0}\frac{K(t,x;\vec{X})}{t}=\left\|  x\right\|
_{X_{1}}.$ Intermediate spaces $\vec{X}_{\theta,q}$ are constructed
by specifying suitable control on the decay of $K(t,x;\vec{X}).$ A
typical construction (Lions-Peetre) can be described as follows: for
$\theta \in(0,1),1\leq q\leq\infty,$ let
$\vec{X}_{\theta,q}=\{x\in\Sigma(\vec
{X}):\left\|  x\right\|  _{\vec{X}_{\theta,q}}<\infty\},$ with%
\[
\left\|  x\right\|  _{\vec{X}_{\theta,q}}=\left\{
\begin{array}
[c]{cc}%
\left\{
\int_{0}^{\infty}(t^{-\theta}K(t,x;\vec{X}))^{q}\frac{dt}{t}\right\}
^{1/q} & \text{if }q<\infty,\\
\sup_{t>0}\left\{  t^{-\theta}K(t,x;\vec{X})\right\}  & q=\infty
\end{array}
\right.  .
\]
Following \cite{jm87} we see that if we write
\[
x=D_{0}(t)x+D_{1}(t)x,\text{ }D_{i}(t)x\in X_{i},
\]
for an optimal decomposition of $x$ for the calculation of (\ref{k1}), then%
\[
K(t,x;\vec{X})=\left\|  D_{0}(t)x\right\|  _{X_{0}}+t\left\|  D_{1}%
(t)x\right\|  _{X_{1}},
\]
and with suitable interpretation for the derivatives
\begin{equation}
\left\|  D_{0}(t)x\right\|
_{X_{0}}=K(t,x)-t\frac{d}{dt}K(t,x;\vec{X})\text{
}a.e.; \label{ai1}%
\end{equation}%
\begin{equation}
\left\|  D_{1}(t)x\right\|
_{X_{1}}=\frac{d}{dt}K(t,x;\vec{X})\text{ }a.e.
\label{ai2}%
\end{equation}

The pair $(L^{1}(\mathbb{R}^{n}),L^{\infty}(\mathbb{R}^{n}))$ is
well understood (cf. \cite{bs}). Without loss of generality we can
assume that $f=\left|  f\right|  .$ An optimal decomposition is then
given by
\[
D_{0}(t)\left|  f\right|  =f_{f^{\ast}(t)},\text{ }D_{1}(t)\left|
f\right| =(\left|  f\right|  -f_{f^{\ast}(t)}),
\]
where for $t>0,$%
\[
f_{f^{\ast}(t)}=\left\{
\begin{array}
[c]{ll}%
\left|  f(x)\right|  -f^{\ast}(t) & \text{if }f^{\ast}(t)<\left|
f(x)\right|
,\\
0 & \text{if }\left|  f(x)\right|  \leq f^{\ast}(t).
\end{array}
\right.  \text{ }%
\]
Therefore the following elementary formulae holds (cf. \cite{bs})
\[
K(t,f;L^{1}(\mathbb{R}^{n}),L^{\infty}(\mathbb{R}^{n}))=\int_{0}^{t}f^{\ast
}(s)ds=tf^{\ast\ast}(t),
\]%
\[
\frac{d}{dt}(K(t,f;L^{1}(\mathbb{R}^{n}),L^{\infty}(\mathbb{R}^{n})))=f^{\ast
}(t).
\]
In this case (\ref{ai1}) and (\ref{ai2}) take the form%
\[
\left\|  D_{0}(t)f\right\|
_{L^{1}}=tf^{\ast\ast}(t)-tf^{\ast}(t),\text{ }t\left\|
D_{1}(t)f\right\|  _{L^{\infty}}=tf^{\ast}(t).
\]
The interaction of gradients and truncation can be quantified here by%
\begin{align}
\left\|  \nabla\left(  f_{f^{\ast}(t)}\right)  \right\|  _{L^{1}}  &
=\int_{\{\left|  f\right|  >f^{\ast}(t)\}}\left|  \nabla\left|
f\right|
\right|  dx\nonumber\\
&  \leq\int_{0}^{t}(\nabla\left|  f\right|  )^{\ast}(s)ds\label{necesita}\\
&  =K(t,\left|  \nabla\left|  f\right|  \right|
;L^{1},L^{\infty}).\nonumber
\end{align}

Using the optimality and the definition of the penalty method
(\ref{k1}) we
readily get the following *reiteration* estimate (cf. \cite{cjm})%
\begin{equation}
K(s,D_{0}(t)f,\vec{X})\geq K(s,f)-s\frac{d}{ds}K(s,f),\quad s\leq
t\quad\text{(a.e.)}. \label{bjc}%
\end{equation}
As a consequence we find%

\begin{equation}
\left\|  D_{0,\vec{X}}(t)f\right\|  _{\vec{X}_{\theta,q}}\geq\int_{0}%
^{t}\left[  \left(
K(s,f;\vec{X})-s\frac{d}{ds}K(s,f;\vec{X})\right)
s^{-\theta}\right]  ^{q}\frac{ds}{s},\text{ }q<\infty, \label{dona1}%
\end{equation}
and%
\begin{equation}
\left\|  D_{0,\vec{X}}(t)f\right\|  _{\vec{X}_{\theta,\infty}}\geq
t^{-\theta
}[K(t,f;\vec{X})-t\frac{d}{dt}K(t,f;\vec{X})]. \label{dona2}%
\end{equation}

Let us see this method in action. We start by rewriting the weak
Gagliardo-Nirenberg inequality using
\begin{equation*}\begin{split}
\left\|  f\right\|  _{M(L^{n^{\prime}})}  &
=\sup_{t>0}\{f^{\ast\ast
}(t)t^{1/n^{\prime}}\}\\
&  =\left\|  f\right\|
_{(L^{1}(R^{n}),L^{\infty}(R^{n}))_{1/n,\infty}}.
\end{split}\end{equation*}
Therefore, for $f\in Lip_{0},$ we have
\[
\left\|  f\right\|  _{(L^{1}(R^{n}),L^{\infty}(R^{n}))_{1/n,\infty}}%
\preceq\left\|  \nabla f\right\|  _{1}.
\]
Inserting the optimal decomposition we get
\begin{equation}
\left\|  D_{0}(t)f\right\|  _{(L^{1}(R^{n}),L^{\infty}(R^{n}))_{1/n,\infty}%
}\preceq\left\|  \nabla D_{0}(t)f\right\|  _{1}. \label{ai4}%
\end{equation}
The left hand side of (\ref{ai4}) can be estimated using
(\ref{dona2}) and we
see that%
\[
\left\|  D_{0}(t)f\right\|  _{(L^{1}(R^{n}),L^{\infty}(R^{n}))_{1/n,\infty}%
}\geq t^{-1/n}[tf^{\ast\ast}(t)-tf^{\ast}(t)].
\]
To estimate the right hand side of (\ref{ai4}) we use
(\ref{necesita}) and we
find%
\[
\left\|  \nabla D_{0}(t)f\right\|  _{1}\leq t\left|  \nabla f\right|
^{\ast\ast}(t).
\]
Altogether we have the familiar%
\[
t^{-1/n}[f^{\ast\ast}(t)-f^{\ast}(t)]\preceq\left|  \nabla f\right|
^{\ast\ast}(t).
\]
If we start with the strong form of the Gagliardo-Nirenberg
inequality, which
in terms of interpolation norms we rewrite as,%
\[
\left\|  f\right\|
_{(L^{1}(R^{n}),L^{\infty}(R^{n}))_{1/n,1}}\preceq\left\| \nabla
f\right\|  _{1},
\]
we can then proceed as above. The only change is that the lower
estimate is
now obtained using (\ref{dona1}) and we find%
\[
\int_{0}^{t}[f^{\ast\ast}(s)-f^{\ast}(s)]s^{-1/n}ds\preceq t\left|
\nabla f\right|  ^{\ast\ast}(t);
\]
a result first derived in \cite{mmp}.

The corresponding inequalities for Log Sobolev inequalities can be
obtained in
analogous manner (cf. \cite{cjm}) but using as starting point%

\begin{equation}
\left\|  f\right\|  _{L(LogL)^{1/2}(R^{n},\gamma_{n})}\preceq\left\|
\nabla
f\right\|  _{L^{1}(R^{n},\gamma_{n})}. \label{auto3}%
\end{equation}

Finally for the connection with the theory of Burkholder-Gundy
\cite{bg} note
that the commutation of the gradient with truncations%
\[
\left|  \nabla f_{t_{1}}^{t_{2}}\right|  =\left|  \nabla f\right|
\chi_{\{t_{1}<\left|  f\right|  <t_{2}\}}%
\]
has the following analog in terms of Square martingale operators,%

\begin{equation}
S(^{\nu}f^{\tau})\leq I(\nu<\tau)S(f), \label{suma1}%
\end{equation}
where $\nu,\tau$ are stopping times, and where $I$ stands for
indicator function. This can be implemented to show an analogue of
(\ref{necesita}).
Consider now the known inequality%
\[
\left\|  Mf\right\|  _{1}\preceq\left\|  Sf\right\|  _{1},
\]
where $M$ is the maximal martingale operator. The method above then
gives the
following inequality due to Herz \cite{he}%
\[
(Mf)^{\ast\ast}(t)-(Mf)^{\ast}(t)\preceq\left(  Sf\right)
^{\ast\ast}(t).
\]
We refer to \cite{cjm} for more details.

\textbf{Acknowledgement} We are very grateful to the referee for
comments and corrections that allowed us to improve the paper.

\bibliographystyle{amsalpha}

\end{document}